\documentclass{amsart}
\usepackage{amssymb}
\usepackage{amscd}
\usepackage{amsfonts}
\usepackage{latexsym}
\usepackage{amsmath}
\usepackage{mathbbol}
\usepackage[all]{xy}
\usepackage{verbatim}

\newtheorem{theorem}{Theorem}[section]
\newtheorem{lemma}[theorem]{Lemma}
\newtheorem{proposition}[theorem]{Proposition}
\newtheorem{corollary}[theorem]{Corollary} 
\theoremstyle{definition}  
\newtheorem{definition}[theorem]{Definition}
\newtheorem{example}[theorem]{Example}
\newtheorem{conjecture}[theorem]{Conjecture}  
\newtheorem{question}[theorem]{Question}
\newtheorem{remark}[theorem]{Remark}

\newcommand{\id}{\text{id}} 
\newcommand{\Id}{\text{id}}

\newcommand{\FPdim}{\text{FPdim}} 
\newcommand{\FP}{\text{FPdim}}
 
\renewcommand{\Vec}{\text{Vec}}

\newcommand{\Hom}{\text{Hom}}

\newcommand{\Rep}{\text{Rep}}

\newcommand{\rev}{\text{rev}}

\newcommand{\sVec}{\text{s}\Vec}

\newcommand{\Fus}{\bf Fus}

\newcommand{\B}{\mathcal{B}}
\newcommand{\C}{\mathcal{C}}
\newcommand{\D}{\mathcal{D}}
\newcommand{\E}{\mathcal{E}}
\newcommand{\F}{\mathcal{F}}
\newcommand{\Z}{\mathcal{Z}}

\newcommand{\A}{\mathcal{A}}

\newcommand{\W}{\mathcal{W}}
\newcommand{\sW}{{s}\mathcal{W}}
\renewcommand{\O}{\mathcal{O}}

\newcommand{\be}{\mathbf{1}}

\newcommand{\g}{\mathfrak{g}}

\renewcommand{\be}{\mathbf{1}}

\newcommand{\bZ}{\mathbb{Z}}

\newcommand{\bt}{\boxtimes}
\newcommand{\bts}{\operatornamewithlimits{\boxtimes}_{\sVec}}
\newcommand{\ot}{\otimes}
\newcommand{\bte}{{\boxtimes}_{\E}}

\newcommand{\kk}{\mathbb{k}}

\newcommand{\bth}{\begin{theorem}}
\renewcommand{\eth}{\end{theorem}}
\newcommand{\bpr}{\begin{proposition}}
\newcommand{\epr}{\end{proposition}}
\newcommand{\ble}{\begin{lemma}}
\newcommand{\ele}{\end{lemma}}
\newcommand{\bco}{\begin{corollary}}
\newcommand{\eco}{\end{corollary}}
\newcommand{\bde}{\begin{definition}}
\newcommand{\ede}{\end{definition}}
\newcommand{\bex}{\begin{example}}
\newcommand{\eeqx}{\end{example}}
\newcommand{\bre}{\begin{remark}}
\newcommand{\ere}{\end{remark}}
\newcommand{\bcj}{\begin{conjecture}}
\newcommand{\ecj}{\end{conjecture}}

\newcommand{\beq}{\begin{equation}}
\newcommand{\eeq}{\end{equation}}

\newcommand{\lb}{\label}
\newcommand{\bpf}{\begin{proof}}
\newcommand{\epf}{\end{proof}}

\begin{document}
\title{On the structure of the Witt group of braided fusion categories}

\author{Alexei Davydov}
\address{A.D.: Department of Mathematics and Statistics,
University of New Hampshire,  Durham, NH 03824, USA}
\email{alexei1davydov@gmail.com}
\author{Dmitri Nikshych}
\address{D.N.: Department of Mathematics and Statistics,
University of New Hampshire,  Durham, NH 03824, USA}
\email{nikshych@math.unh.edu}
\author{Victor Ostrik}
\address{V.O.: Department of Mathematics,
University of Oregon, Eugene, OR 97403, USA}
\email{vostrik@math.uoregon.edu}

\date{\today}
\maketitle  

\begin{section}{Introduction}

In our previous work \cite{DMNO} (written jointly with M.~M\"uger) we introduced the Witt group
$\W$  of non-degenerate braided fusion categories over an algebraically closed field $\kk$
of characteristic zero. Its definition is similar to that of the classical
Witt group $\W_{pt}$ of finite abelian groups endowed with a non-degenerate quadratic form 
with values in $\kk^\times$ (such groups correspond to
pointed braided fusion categories).  Namely, $\W$ is obtained as the quotient of the monoid
of non-degenerate braided fusion categories by the submonoid of Drinfeld centers.
In particular, $\W$ contains $\W_{pt}$  as a subgroup.  

The group $\W_{pt}$ is explicitly known, see Section~\ref{Prelim-Witt}.  One would like to
have a similarly explicit description of the more complicated categorical Witt group $\W$. 
In \cite{DMNO} a number of questions regarding the structure of $\W$ was asked.  The goal
of the present paper is to answer some of these questions and further analyze the structure of $\W$.

The key observation used in this paper is that one can define the Witt group $s\W$ 
of {\em slightly degenerate} braided fusion categories and a group homomorphism $S: \W \to s\W$
whose kernel is the cyclic group $\mathbb{Z}/16\mathbb{Z}$ generated by the Witt
classes of Ising categories, see Section~\ref{homS}. 
 
It was pointed to us by V.~Drinfeld that
completely anisotropic slightly degenerate braided fusion categories admit a canonical
decomposition into a tensor product (over the category $\sVec$ of super vector spaces) 
of simple categories, see  Theorem~\ref{decomposition theorem}.  
Furthermore, there are very few  relations between the classes of simple categories in  $s\W$,
see Corollary~\ref{relations in sW}.  This allows to obtain a rather explicit description of the 
structure  of  $s\W$ (and, consequently, of  $\W$). Namely, we have 
\begin{equation}
s\W =  s\W_{pt} \bigoplus s\W_2 \bigoplus s\W_\infty,
\end{equation}
where $s\W_{pt}$ denotes the subgroup of $s\W$
generated by the Witt classes of  slightly degenerate pointed braided fusion categories,
$s\W_2$ is an elementary Abelian $2$-group, and the subgroup
$s\W_\infty$ is a free Abelian group of countable rank,
see Proposition~\ref{sW decomposition}.   In particular, groups $\W$ and $\sW$ are $2$-primary,
i.e., have no odd torsion.

The organization of the present paper can be  summarized as follows. 

Section~\ref{Preliminaries} contains a necessary background
material about fusion categories and definition of the Witt group from \cite{DMNO}.  Here we also
introduce the notion of tensor product of  braided fusion categories over  a symmetric category
(Section~\ref{Prelim-over E}). 

In Section~\ref{Etale algebras in products} we give a complete classification of \'etale algebras
in tensor products of braided tensor categories.  
The main result of this Section is Theorem~\ref{etale in CD}.

In Section~\ref{Tensor product decomposition} we establish a tensor product decomposition 
of a slightly degenerate braided fusion category without non-trivial Tannakian subcategories
into a tensor product of simple factors. This is done in  Theorem~\ref{decomposition theorem}.  

Finally, Section~\ref{Witt section} contains the definition of the Witt group of non-degenerate braided
fusion categories over a symmetric fusion category $\E$. The principal objects of our study
groups $\W$ and $s\W$ correspond to $\E=\Vec$ and $\sVec$, respectively. 
Here we describe the structure of $\sW$ (Section~\ref{structure}) and describe all relations
between the Witt classes of categories $\C(sl(2),\, k),\, k\geq 1$. 

\textbf{Acknowledgments} We are deeply grateful to Vladimir Drinfeld for sharing with us his ideas
regarding the structure of slightly degenerate categories.  The statement of Theorem~\ref{decomposition theorem}
is due to him.  We are thankful to Michael M\"uger  for  collaboration on \cite{DMNO}  (of which
this paper is a continuation) and for   many useful discussions about algebras in tensor
products of braided categories that led to results of Section~\ref{Etale algebras in products}.  
We also thank Terry Gannon for his insights about algebras in braided categories
associated to affine Lie algebras.   The work of D.N.\ was partially supported by the NSF grant 
DMS-0800545. The work of V.O.\ was  partially supported by the NSF grant DMS-0602263.

\end{section}

\begin{section}{Preliminaries}
\label{Preliminaries}

Throughout this paper our base field $\kk$ is an algebraically closed field of characteristic zero.

\subsection{Fusion categories}
\label{Prelim-fuscat}

By definition (see \cite{ENO1}), a {\em fusion category} over $\kk$
is a $\kk-$linear semisimple rigid tensor category with finitely many simple objects and
finite dimensional spaces of morphisms such that 
its unit object $\be$ is simple.  By a {\em fusion subcategory} of a fusion category 
we always mean a full tensor subcategory closed under taking of direct summands. 
Let $\Vec$ denote the fusion category of finite dimensional vector spaces over $\kk$.
Any fusion category $\A$ contains a trivial fusion subcategory consisting of multiples
of $\be$. We will  identify  this subcategory with $\Vec$. A fusion category $\A$
is called {\em simple} if $\Vec$ is the only proper fusion subcategory of $\A$.

If $\A$ and $\B$ are fusion subcategories of a fusion category $\C$ we will denote by 
$\A \vee \B$ the fusion subcategory of $\C$ generated by $\A$ and $\B$, that is smallest
fusion subcategory which contains both $\A$ and $\B$.

A fusion category is called {\em pointed} if all its simple objects are invertible.
For a fusion category $\A$ we denote  $\A_{pt}$ the maximal pointed fusion subcategory
of $\A$. 

We will denote $\A \boxtimes \B$ the tensor product of fusion categories $\A$ and $\B$.


For a  fusion category $\A$ we denote by $\O(\A)$ the set of
isomorphism classes of simple objects in $\A$.

Let $\A$ be a fusion category and let $K(\A)$ be its Grothendieck ring. 
There exists a unique ring
homomorphism $\FPdim: K(\A)\to \mathbb{R}$ such that $\FPdim(X)>0$ for any $0\ne X\in \A$, see
\cite[Section 8.1]{ENO1}. For a fusion category $\A$ one defines (see \cite[Section  8.2]{ENO1}) its
{\em Frobenius-Perron dimension}:
\begin{equation}
\label{FPdim def}
\FPdim(\A)=\sum_{X\in \O(\A)}\,\FPdim(X)^2.
\end{equation}

Let $\A_1,\, \A_2$ be fusion categories such that $\FPdim(\A_1) =\FPdim(\A_2)$.
By \cite[Proposition 2.19]{EO} a tensor functor $F: \A_1 \to \A_2$ is an equivalence
if and only if it is injective and if and only if it is surjective. 

Here we call $F$ {\em injective} if it is fully faithful and {\em surjective}
if every object $Y$ in $\A_2$ is a subobject of $F(X)$ for some $X$ in $\A_1$.

\subsection{Braided fusion categories}
\label{Prelim-brcat}

A {\em braided} fusion category  is a fusion category $\C$
endowed with a braiding $c_{X,Y}: X\ot Y\xrightarrow{\sim} Y\ot X$, see \cite{JS}. 
For a braided fusion category its {\em reverse} $\C^{\rev}$ 
is the same fusion category with a new braiding
$\tilde{c}_{X,Y}=c_{Y,X}^{-1}$. 

A braided fusion category is {\em symmetric} if
$\tilde{c} = c$.  A symmetric  fusion category $\C$ is {\em Tannakian} if there is a finite 
group $G$ such that  $\C$ is equivalent to the category $\Rep(G)$ of finite dimensional
representations of $G$  as a braided fusion category.  It is known \cite{D2} that $\C$ is Tannakian
if and only if there is a braided tensor functor $\C \to \Vec$.

An example of a non-Tannakian symmetric fusion category is the category $\sVec$
of super vector spaces. 

Recall from \cite{Mu2} that
objects $X$ and $Y$ of a braided fusion category $\C$ are said to
{\em centralize\,}  each other if
\begin{equation} \label{monodromy-drinf}
c_{Y,X}\circ c_{X,Y} =\id_{X\ot Y}.
\end{equation}
The {\em centralizer\,} $\D'$ of a fusion subcategory $\D\subset\C$ is defined to
be the full subcategory of objects of $\C$ that centralize each object of $\D$.
It is easy to see that $\D'$ is a fusion subcategory of $\C$.
Clearly, $\D$ is symmetric if and only if $\D\subset\D'$. The {\em symmetric center} of $\C$
is its self-centralizer $\C'$.

A  braided fusion category $\C$ is called {\em non-degenerate} if $\C'=\Vec$.
If $\C$ has a spherical structure then $\C$ is non-degenerate if and only if it is modular
\cite[Proposition 3.7]{DGNO1}.

For a fusion subcategory $\D$ of a braided fusion category $\C$ one 
has the following properties (see \cite[Corollary 3.11 and Theorem 3.14]{DGNO1}):
\begin{gather}
\label{D''} \D''=\D \vee \C' ,\\
\label{prodFP}
\FPdim(\D)\FPdim(\D')=\FPdim(\C) \FPdim(\D \cap \C').
\end{gather}

Any pointed braided fusion category comes from a {\em pre-metric group}, i.e.,  
a finite Abelian group $A$ equipped with a quadratic form $q:A \to \kk^\times$, 
see \cite[Section 3]{JS} and \cite[Section 2.11]{DGNO1}. 
Let $\C(A,\, q)$ denote the corresponding braided fusion category.

By an {\em Ising} braided fusion category we understand a non-pointed
braided fusion category of the Frobenius-Perron dimension $4$.  Such a category
has $3$ classes of simple objects of the Frobenius-Perron dimension $1$, $1$, and $\sqrt{2}$.
An Ising braided fusion category $\mathcal{I}$ is always non-degenerate and $\mathcal{I}_{pt} \cong \sVec$.
It is known that there are $8$ equivalence classes of Ising braided fusion categories.
See \cite[Appendix B]{DGNO1} for a detailed discussion  of Ising categories.

The following notion was introduced in \cite[Section 2.10]{ENO2}. It plays a crucial role in this paper.

\begin{definition}
\label{def sldeg}
A braided fusion category $\C$ is called {\em slightly degenerate} if $\C' =\sVec$.
\end{definition}

Equivalently, $\C$ is slightly degenerate if its symmetric center $\C'$ is non-trivial and 
contains no non-trivial Tannakian subcategories. 

Let $\C$ be a slightly degenerate
braided fusion category and let  $\delta$ denote the simple object generating $\C'$. 
Then $\delta \ot X \not\cong X$ for any simple $X$ in $\C$, see \cite[Lemma 5.4]{Mu1}.
If $\C$ is a pointed slightly degenerate braided fusion category then $\C \cong \C_0 \bt \sVec$,
where $\C_0$  is a non-degenerate braided fusion category, see \cite[Proposition 2.6(ii)]{ENO2}
or \cite[Corollary A.19]{DGNO1}.
 
\begin{example}
\label{slight examples}
There are several general ways to construct examples of slightly degenerate
braided fusion categories.
\begin{enumerate}
\item[(i)]
Let $\C$ be a non-degenerate braided fusion category equipped with a braided tensor functor
$\sVec \to \C$. Then the centralizer of  the image of $\sVec$ in $\C$ is a slightly degenerate
braided fusion category.
\item[(ii)] Let $\C$ be a non-degenerate braided fusion category. Then $\C \bt \sVec$
is a  slightly degenerate braided fusion category.
\end{enumerate}
\end{example}

\subsection{Drinfeld center of a fusion category} 
\label{dcenter}

For any fusion category $\A$ its {\em Drinfeld center}
$\mathcal{Z}(\A)$ is defined as the category whose objects are
pairs $(X, \gamma_X)$, where $X$ is an object of $\A$ and
$\gamma_X : V\ot X
\simeq X \ot V$, $V\in\A,$
is a natural family of isomorphisms, satisfying well known
compatibility conditions.
It is known that $\Z(\A)$ is a non-degenerate braided fusion
category, see \cite[Corollary 3.9]{DGNO1}.  We have (see 
\cite[Theorem 2.15, Proposition 8.12]{ENO1}):
\begin{equation}
\label{dimZ}
\FPdim(\Z(\A))=\FPdim(\A)^2.
\end{equation}


For a braided fusion category $\C$ there are two braided functors 
\begin{eqnarray}
\C \to \Z(\C) &:& X\mapsto (X,\, c_{-,X}),\\
\C^{\rev} \to \Z(\C) &:& X\mapsto  (X,\, \tilde c_{-,X}).
\end{eqnarray}
These functors are injective and so we can identify
$\C$ and $\C^{\rev}$ with their images in $\Z(\C)$. These images 
centralize each other, i.e., $\C' =\C^{\rev}$.  This allows
to define a braided tensor functor 
\begin{equation}
\label{Gfun}
G: \C \boxtimes \C^{\rev}\to \Z(\C).
\end{equation}
It was shown in \cite{Mu3} and \cite[Proposition 3.7]{DGNO1}
that $G$ is a braided equivalence if and only if $\C$ is
non-degenerate. 

Let $\C$ be a braided fusion category and let $\A$ be a fusion
category. Let $F: \C \to \A$ be a tensor functor. 

\begin{definition}
\label{def centfun} 
A structure of a {\em central functor}
on $F$ is a braided tensor functor $\C\to\Z(\A )$ together with isomorphism between its composition 
with the forgetful functor $\Z(\A)\to\A$ and $F$.
\end{definition}

Equivalently, a structure of central functor on $F$ is a natural family of isomorphisms
$Y\ot F(X) \xrightarrow{\sim} F(X)\ot Y$, $X\in\C$, $Y\in\A$, satisfying certain compatibility conditions,
see \cite[Section  2.1]{Be}.

\subsection{\'Etale algebras in braided fusion categories}
\label{Prelim-etale}

Here we recall definition and basic properties of \'etale algebras in braided
fusion categories, see \cite[Section 3]{DMNO} for details.

Let $\A$ be a fusion category. In this paper an {\em algebra} $A\in \A$
is an associative algebra with unit, see e.g. \cite[Definition 3.1]{O1}.
An algebra $A\in \A$ is said to be {\em separable} if the
multiplication morphism $m: A\ot A\to A$ splits as a morphism of $A$-bimodules.

Let $\C$ be a braided fusion category.
An algebra $A$ in $\C$ is called  {\em \'etale} if
it is both commutative and separable.
We say that an \'etale algebra $A\in \C$ is {\em connected} if $\dim_\kk \Hom_\C(\be,\, A)=1$.

\begin{example}
\label{regular etale}
Let $\E = \Rep(G) \subset \C$ be a Tannakian subcategory of $\C$. 
The algebra  $\mbox{Fun}(G)$ of functions on $G$  is a connected \'etale algebra in $\C$.
\end{example}

Let $A$ be an  \'etale algebra in $\C$. An \'etale algebra {\em over} $A$ is an \'etale algebra
in $\C$ containing $A$ as a subalgebra.  For a fusion subcategory $\D \subset \C$ we denote
by $A\cap \D$ the maximal \'etale subalgebra of $A$ contained in  $\D$.

There is a canonical correspondence between \'etale algebras in $\C$ and surjective central 
functors from $\C$.  Namely, let $\A$ a fusion category, and let $F: \C \to \A$
be a central functor. Let $I : \A \to \C$ denote  the right adjoint functor of $F$. Then the object
$A=I(\be)\in \C$ has a canonical structure of connected \'etale algebra, see \cite[Lemma 3.5]{DMNO}.

Conversely, let $A\in \C$ be a connected \'etale algebra. Let $\C_A$ denote the category of right 
$A$-modules in $\C$.  Then the free module functor
\[
F_A: \C \to \C_A,\quad X \mapsto X \ot A
\] 
is surjective and has a canonical structure of a central functor.  The above constructions (of \'etale
algebras from central functors and vice versa) are inverses of each other.  It was shown
in \cite[Lemma 3.11]{DMNO} that
\begin{equation}
\label{FPdimCA}
\FPdim(\C_A)=\frac{\FPdim(\C)}{\FPdim(A)}.
\end{equation}

Note that every right $A$-module $M$ can be made an $A$-bimodule in two ways, using the braiding $c$
of $\C$  and  its reverse $\tilde{c}$. We call $M$  {\em dyslectic} (or {\em local}) if two resulting bimodules are equal.
The category $\C_A^0$ of dyslectic $A$-modules has a canonical structure of a braided fusion category.
If $\C$ is non-degenerate then so is $\C_A^0$ and we have 
\begin{equation}
\label{FPdimCA0}
\FPdim(\C_A^0)=\frac{\FPdim(\C)}{\FPdim(A)^2},
\end{equation}
see \cite[Corollary 3.32]{DMNO}.

A {\em Lagrangian algebra} in a non-degenerate braided fusion category $\C$ is 
a connected  \'etale algebra $A\in \C$ such that $\FP(\C)=\FP(A)^2$, see \cite[Definition 4.6]{DMNO}. 
A Lagrangian algebra in $\C$ exists if and only if $\C \cong \Z(\A)$ for some fusion category 
$\A$. Moreover,
the isomorphism classes of  Lagrangian algebras in $\Z(\A)$  are in bijection 
with equivalence classes of  indecomposable $\A$-module categories, 
see \cite[Theorem 1.1]{KR} and \cite[Proposition 4.8]{DMNO}.

Following \cite{DMNO} we say that a braided fusion category $\C$
is {\em completely anisotro\-pic}  if the only connected \'etale algebra $A\in \C$ is  $A =\be$.  
By Example \ref{regular etale}, a completely anisotropic category 
has no Tannakian subcategories other than $\Vec$ and so is either non-degenerate
or slightly degenerate. 
By \cite[Lemma 5.12]{DMNO} any central functor from a completely anisotropic
braided tensor category is injective.

\subsection{Fusion categories over a symmetric category}
\label{Prelim-over E}

The following definition was given in \cite[Definition 4.16]{DGNO1}.

\begin{definition}  
\label{over}
Let $\E$ be a symmetric fusion category.
A {\em fusion category over $\E$}  is a fusion category $\A$
equipped with a braided tensor functor $T: \E \to \Z(\A)$.
\end{definition}

In this paper we will only consider fusion categories $\A$ over $\E$
with the property that the composition of $T$ with the forgetful functor 
$\Z(\A)\to \A$ is fully faithful. Note that this property is automatically
satisfied for  $\E=\sVec$.  

To define 
tensor functors over $\E$ we need the following auxiliary construction, generalizing the 
notion  of the Drinfeld center of a fusion category, see Section~\ref{dcenter}.

Let $F:\A\to\B$ be a tensor functor. Its {\em relative center} $\Z(F)$ is the category of pairs $(Z,z)$, where $Z$ is an object of $\B$ 
and $z_X:Z\ot F(X)\xrightarrow{\sim}  F(X)\ot Z$ is a collection of isomorphisms, natural in $X\in\A$, and such that the following diagram 
$$\xymatrix{ Z\ot F(X\ot Y) \ar[d]_{F_{X,Y}}  \ar[rr]^{z_{X\ot Y}} && F(X\ot Y)\ot Z \ar[d]^{F_{X,Y}} \\ 
Z\ot F(X)\ot F(Y) \ar[dr]^{z_X} && F(X)\ot F(Y)\ot Z\\
& F(X)\ot Z\ot F(Y) \ar[ru]^{z_Y} }$$
commutes for all $X,Y\in\A$.  
Here $F_{X,Y}:F(X\ot Y) \xrightarrow{\sim}  F(X)\ot F(Y)$ is the tensor structure of $F$.
Here and below we suppress all identity morphisms
and  associativity constraints.

The category $\Z(F)$ is tensor with respect to the tensor product given by 
\[
(Z,z)\ot(W,w) = (Z\ot W,z\ot w), 
\]
where $(z\ot w)_X$ is the composition
$$\xymatrix{Z\ot W\ot F(X) \ar[r]^{w_X} & Z\ot F(X)\ot W \ar[r]^{z_X} & F(X)\ot Z\ot W}.$$

\begin{remark}
Let $\B \subset \A$ be a fusion subcategory and $F: \B \hookrightarrow \A$ be an embedding.
Then $\Z(F)$ is the relative center of $\B$ in $\A$ \cite{Ma}. In particular,
the relative center  $\Z(\id_\A)$ of the identity functor $\id_A:\A\to\A$ is $\Z(\A)$. 
\end{remark}

Clearly, the forgetful functor $\Z(F)\to \B$ is tensor. 

There are two canonical tensor functors $\Z(\A)\to\Z(F), \Z(\B)\to\Z(F)$ which fit into a commutative diagram of tensor functors:
$$\xymatrix{\Z(\A) \ar[d] \ar[r] & \Z(F) \ar[rd] & \Z(\B) \ar[l] \ar[d] \\ \A \ar[rr]^F && \B }$$

\begin{definition}  
\label{fover}
A tensor functor $F:\A\to\B$ between fusion categories over $\E$ is called a {\em tensor functor over $\E$} if the following diagram of tensor functors commutes up to a tensor isomorphism:
$$\xymatrix{\E \ar[r]^{T_\A} \ar[d]_{T_\B} & \Z(\A) \ar[d] \\ \Z(\B) \ar[r] & \Z(F) }.$$
\end{definition}
In other words, a structure of a tensor functor over $\E$ on $F:\A\to\B$ is a tensor isomorphism $u_E:F(T_\A(E))\to T_\B(E),\, E\in\E,$ such that the diagram
$$
\xymatrix{F(T_\A(E)\ot X) \ar[d] \ar[rr]^{F_{T(E),X}} && F(T_\A(E))\ot F(X) \ar[rr]^{u_X} && T_\B(E)\ot F(X) \ar[d] \\ 
F(X\ot T_\A(E)) \ar[rr]^{F_{X,T(E)}} && F(X)\ot F(T_\A(E)) \ar[rr]^{u_X} && F(X)\ot T_\B(E) }
$$ 
commutes. Here the vertical arrows are the central structures of $T_\A(E)$ and $T_\B(E)$, respectively. 

A tensor natural transformation $a:F\to G$ between tensor functors over $\E$ is a {\em tensor natural transformation over $\E$} if the diagram
$$
\xymatrix{F(T_\A(E)) \ar[rr]^{a_{T(E)}} \ar[rd]_{u_E} && G(T_\A(E))\ar[ld]^{v_E} \\ & T_\B(E) }
$$
commutes for any $E\in\E$. Here $u_E$ and $v_E$ denote the structures of functors over $\E$ for $F$ and $G$, respectively.

Note that relative centers  of successive tensor functors $F:\A\to\B, G:\B\to\C$ are related to the 
relative center of the composition by a pair of canonical tensor functors $\Z(F)\to\Z(G\circ F), \Z(G)\to\Z(G\circ F)$ 
fitting into a commutative diagram of tensor functors:  
$$\xymatrix{ && \Z(G\circ F) \ar[dddrr]\\ & \Z(F) \ar[ru]  \ar[rdd] && \Z(G) \ar[rdd] \ar[lu] \\
\Z(\A) \ar[d] \ar[ru] && \Z(\B) \ar[d] \ar[ru] \ar[lu] && \Z(\C) \ar[d] \ar[lu] \\ \A \ar[rr]^F && \B \ar[rr]^G && \C}$$
This allows us to compose tensor functors over $\E$. In other words, the structure of tensor functor over $\E$ on the composition $G\circ F$ is given by 
$$\xymatrix{G(F(T_\A(E))) \ar[r]^{G(u_E)} & G(T_\B(E)) \ar[r]^{v_E} & T_\C(E) }$$

Thus we have a 2-category $\Fus_\E$ of fusion categories, functors and natural transformations over a symmetric fusion category $\E$. 

Let us explain  the  functoriality of $2$-categories $\Fus_\E$ with respect to braided tensor functors $\E\to\F$.
Construct a {\em base change} 2-functor
\begin{equation}
\label{base change}
-\boxtimes_\E\F:\Fus_\E\to \Fus_\F,\quad \A\mapsto \A\boxtimes_\E\F.
\end{equation}
Here the tensor product $\A\boxtimes_\E\F$  of fusion categories over $\E$  is defined as follows. 
Consider a composition of braided  tensor functors 
$$
\xymatrix{\E\boxtimes\F\ar[r] & \F\boxtimes\F \ar[r]^\ot & \F}
$$ 
where the first is given by the braided functor $\E\to\F$ and the second is the tensor product of $\F$. Denote by $R$ the right adjoint to the composition. The value $R(\be)$ on the unit object is a connected \'etale algebra in 
$\E\boxtimes\F$. Denote by $A$ its image in $\A\boxtimes\F$. Now $\A\boxtimes_\E\F$ is the category $(\A\boxtimes\F)_A$ of $A$-modules in $\A\boxtimes\F$.

Note that Deligne tensor product of fusion categories induces a 2-functor
$$\Fus_\E\times\Fus_\F\to\Fus_{\E\boxtimes\F},\quad (\A,\B)\mapsto \A\boxtimes\B.$$
Assume that $\F=\E$ and compose the above 2-functor with the base change 2-functor $$-\boxtimes_{\E\boxtimes\E}\E:\Fus_{\E\boxtimes\E}\to \Fus_\E$$ induced by the tensor product functor $\E\boxtimes\E\to\E$ (which is a braided monoidal functor). This gives a 2-functor (the {\em Deligne tensor product over $\E$})
$$\Fus_\E\times\Fus_\E\to\Fus_{\E},\quad (\A,\B)\mapsto \A\boxtimes_\E\B$$ which turns $\Fus_\E$ into a monoidal 2-category.

\begin{remark}
More explicitly, we define $\A\boxtimes_\E\B$ as the category of (right) $A$-modules $(\A\boxtimes\B)_A$, where $A$ is an \'etale algebra in $\Z(\A\boxtimes\B)$ defined as $(T_\A\boxtimes T_\B)(R(\be))$, with $R:\E\to\E\boxtimes\E$ being the right adjoint functor to the tensor product functor $\ot:\E\boxtimes\E\to\E$. 

Note that our construction gives the tensor product of fusion categories over a symmetric category defined by 
Greenough in \cite[Section 6]{G}. 
\end{remark}

The base change 2-functors preserve tensor product, i.e., are monoidal:
$$(\A\bt_\E\B)\bt_\E\F \simeq (\A\bt_\E\F)\bt_\F(\B\bt_\E\F).$$ 

\begin{definition}  
\label{brover}
A {\em braided  fusion category over $\E$}  or {\em braided fusion $\E$-category} 
is a braided fusion category $\C$ equipped with a braided tensor embedding $T: \E \to \C'$.   
An {\em $\E$-subcategory}  of a braided fusion $\E$-category is its  fusion subcategory
containing $\E$. 
\end{definition}

A braided  fusion category over $\E$ can be seen as a 
fusion category over $\E$ with a braided functor $T_\C:\E\to\Z(\C)$ defined as the composition of braided functors
$$\xymatrix{\E \ar[r] & \C' \ar[r] & \C \ar[r] & \Z(\C)}$$

Let $\C$ and $\D$ be braided 
fusion categories over $\E$. Since the functors $T_\C, T_\D$ factor through symmetric centers $\C', \D'$ respectively the algebra $A = (T_\C\boxtimes T_\D)(R(\be))$ is an object of the symmetric center $(\C\boxtimes\D)'$. Thus the category of modules $(\A\boxtimes\B)_A$ coincides with its subcategory of local modules $(\A\boxtimes\B)^0_A$ and hence is braided. In other words the Deligne tensor product $\C\boxtimes_\E\D$ of two braided fusion categories over $\E$ is a braided fusion category over $\E$.

\bre
\label{over E = Amod}
Recall from \cite[section 5]{JS} that a braiding on a fusion category $\C$ is equivalent to a monoidal structure on the tensor product functor $\ot:\C\boxtimes\C\to \C$. 
In slightly more general and abstract terms, the 2-category $\Fus^{br}$ of braided fusion categories is bi-equivalent to the 2-category of {\em pseudo-monoids} in $\Fus$. 
\newline
The above statement can be generalized to fusion categories over $\E$. It is not hard to see that the structure of a braided fusion category over $\E$ on $\C$ is equivalent to a monoidal structure on the tensor product functor $\ot:\C\boxtimes_\E\C\to \C$. 
Again reformulating it in slightly more general and abstract terms, the 2-category $\Fus^{br}_\E$ of braided fusion categories is bi-equivalent to the 2-category of pseudo-monoids in $\Fus_\E$. 

This in particular allows us to extend base change
to braided fusion categories over $\E$. More precisely, for a braided fusion
category $\C$ over a symmetric $\E$ and for a symmetric tensor functor
$\E\to\F$ the base change $\C\boxtimes_\E\F$ has a natural structure
of a braided fusion category over $\F$.
\ere

\subsection{De-equivariantization}

Let $\A$ be a fusion category and let $\E =\Rep(G)$ be a Tannakian subcategory of  $\Z(\A)$
which embeds into $\A$ via the forgetful functor $\Z(\A)\to \A$.  
Then $\A$ is a fusion category over $\E$.
The {\em de-equivariantization} of  $\A$  is the fusion category $\A \bt_\E \Vec$
obtained from $\A$ by means of the base change \eqref{base change}
corresponding to the braided fiber functor $\E \to \Vec$. 

Explicitly,   $\A \bt_\E \Vec = \Rep_\A(A)$
where $A=\mbox{Fun}(G)$ is  the regular algebra of $\E$ (viewed as
an \'etale  algebra in $\Z(\A)$), see Example~\ref{regular etale} and $\Rep_\A(A)$
is defined in \cite[Definition 3.16]{DMNO}

There is a canonical surjective 
tensor functor $\A \to \A \bt_\E \Vec$ assigning to each $X\in \A$ the free
$A$-module $X\ot A$. 

The above construction extends to braided fusion categories as follows.
Let $\E \subset \C'$ be a Tannakian subcategory.
Then $\C$ is a braided fusion category over $\E$, 
its de-equivariantization $\C \bt_\E \Vec$ is a braided fusion category, 
and the canonical central tensor functor $\C \to \C \bt_\E \Vec$ is braided.

We refer the reader to \cite[Section 4]{DGNO1} for details.

\begin{remark}
\label{tensor over E = equiv}
Tensor product of fusion categories $\A$ and $\B$ over a symmetric category $\E$ 
defined in Section~\ref{Prelim-over E} is a special case of de-equivariantization.
Indeed, let $\F$ be the Tannakian subcategory of $\E \bt \E$ such that
$(\E \bt \E) \bt_\F \Vec \cong \E$  (existence of such a subcategory  $\F$
follows from \cite[Corollary 3.22]{DMNO} since the tensor
product functor $\ot: \E\bt \E \to \E$ is braided).
Then
\[
\A\bt_\E \B  \cong (\A\bt \B) \bt_\F \Vec.
\]
The \'etale algebra $A\in \Z(\A\bt \B)$ 
constructed in Remark~\ref{over E = Amod}  identifies with the regular algebra of $\F$.
\end{remark}

\begin{proposition}
\label{Center of the de-equivariantization}
Let $\A$ be a fusion category over a Tannakian category $\E$ such that
the functor $\E \to \Z(\A)$ is an embedding.  Then
$\Z(\A \bt_\E \Vec) \cong  \E' \bt_\E \Vec$. (Here we identify $\E$
with its image in $\Z(\A)$).
\end{proposition}

Proposition~\ref{Center of the de-equivariantization}  is proved in \cite[Proposition 2.10]{ENO2}. 

\begin{proposition}
\label{De-equivariantization commutes with centralizers}
Let $\C$ be a braided fusion category, let $\E \subset \C'$
be a Tannakian subcategory, and let $\D \subset \C$ be a fusion subcategory
containing $\E$.  The braided fusion category
$\D' \bt_\E \Vec$ is equivalent to the centralizer of $\D \bt_\E \Vec$  in $\C \bt_\E \Vec$.
In other words, de-equivariantization of braided fusion categories 
commutes with taking the centralizers.
\end{proposition}

Proposition  \ref{De-equivariantization commutes with centralizers}
is proved in  \cite[Proposition 4.30]{DGNO1}.

\subsection{The Witt group of non-degenerate braided fusion categories}
\label{Prelim-Witt}

The {\em Witt group}  $\W$ of non-degenerate braided fusion categories was defined in \cite{DMNO}. 

Two non-degenerate braided fusion categories $\C_1$ and $\C_2$ are {\em Witt equivalent}
if there exist fusion categories $\A_1$ and $\A_2$ such that  $\C_1\bt \Z(\A_1)\cong
\C_2\bt \Z(\A_2)$.  The elements of $\W$ are Witt equivalence classes of non-degenerate braided 
fusion categories.   The group operation of $\W$ is given by the Deligne tensor product $\bt$.
Let $[\C]$ denote the Witt equivalence class containing category $\C$.  The unit object of $\W$ is $[\Vec]$
and the inverse of $[\C]$ is $[\C^\rev]$. 

The following results were established in \cite[Section 5]{DMNO}.
For any \'etale algebra $A\in \C$ we have $[\C]=[\C_A^0]$. The Witt class of $\C$ is trivial, i.e.,  
$[\C]=[\Vec]$, if and only if $\C \cong \Z(\A)$ for some fusion category $\A$.  Every Witt class
contains a unique (up to equivalence) completely anisotropic representative. 

The Witt group $\W$ contains the  subgroup $\W_{pt}$ of the Witt  classes of non-degenerate pointed braided
fusion categories  \cite[Section 5.3]{DMNO}. The latter coincides with the classical Witt
group of metric groups (i.e., finite abelian groups equipped with a non-degenerate
quadratic form). The group $\W_{pt}$ is explicitly known, see e.g., \cite[Appendix A.7]{DGNO1}. Namely, 
\begin{equation}
\label{Wpt}
\W_{pt}=\bigoplus_{p\text{ is prime}}\, \W_{pt}(p),
\end{equation}
where $\W_{pt}(p)\subset \W_{pt}$ consists of the classes of metric $p-$groups. 

The group $\W_{pt}(2)$ is isomorphic to $\mathbb{Z}/8\mathbb{Z} \oplus
\mathbb{Z}/2\mathbb{Z}$; it is generated by two classes $[\C(\mathbb{Z}/2\mathbb{Z},\,q_1)]$ and $[\C(\mathbb{Z}/4\mathbb{Z},\,q_2)]$, where 
$q_1,\, q_2$ are any non-degenerate forms. For $p\equiv 3 \,(\mathrm{mod}\,4)$ we have $\W_{pt}(p)\cong \mathbb{Z}/4\mathbb{Z}$ 
and the class $[\C(\mathbb{Z}/p\mathbb{Z},\,q)]$ is a generator for any non-degenerate form $q$. For 
$p\equiv 1 \,(\mathrm{mod}\, 4)$
the group $\W_{pt}(p)$ is isomorphic to $\mathbb{Z}/2\mathbb{Z} \oplus \mathbb{Z}/2\mathbb{Z}$;  it is generated by the two classes 
$[\C(\mathbb{Z}/p\mathbb{Z},\,q')]$ and $[\C(\mathbb{Z}/p\mathbb{Z},\,q'')]$  with $q'(l)=\zeta^{l^2}$ and $q''(l)=\zeta^{nl^2}$, where
$\zeta$ is a primitive $p$th root of unity in $\kk$ and 
$n$ is any quadratic non-residue modulo $p$.

It was also explained in  \cite[Section 6.4]{DMNO} that $\W$ contains a cyclic subgroup $\W_{Ising}$ of order $16$
generated by the Witt classes of Ising braided fusion categories. This group can explicitly be described 
as follows. For every Ising braided category $\mathcal{I}$ the class $[\mathcal{I}]$ is a generator of $\W_{Ising}$.
The unique index $2$ subgroup  of $\W_{Ising}$  consists of the Witt classes of categories
$\C(A, \, q)$, where $(A,q)$ is a metric group of order 4  such that there exists $u\in A$ with $q(u)=-1$
(cf. \cite[\S A.3.2]{DGNO1}).

\subsection{Tensor categories from affine Lie algebras}
\label{Prelim-affine}

Let $\g$ be a finite dimensional simple Lie algebra and let $\hat \g$ be the
corresponding affine Lie algebra. For any $k\in \mathbb{Z}_{>0}$ let $\C(\g,k)$ be the
category of highest weight integrable $\hat \g-$modules of level $k$, see e.g. \cite[Section~ 7.1]{BaKi} 
where this category is denoted $\O_k^{int}$. The category $\C(\g,k)$ can be identified with
the category $\Rep(V(\g,k))$, where $V(\g,k)$ is the simple
vertex operator algebra (VOA) associated with the vacuum 
$\hat \g-$module of level $k$, {\em the affine VOA}. Thus the category $\C(\g,k)$ has a structure of 
modular tensor category, see \cite{HuL}, \cite[Chapter 7]{BaKi}. In particular the category $\C(\g,k)$ is braided and non-degenerate. 

\begin{example} The category $\C(sl(n),1)$ is pointed. 
It identifies with $\C(\bZ/n\bZ,q)$ where $q(l)=e^{\pi il^2\frac{n-1}n}$, $l\in \bZ/n\bZ$.
More generally, for a simply laced $\g$ the category $\C(\g,1)$ is pointed \cite{fk}. 
\end{example}

The following formula for the central
charge is very useful, see e.g. \cite[7.4.5]{BaKi}:
\begin{equation} \label{cformula}
c(\C(\g,k))=\frac{k\dim \g}{k+h^\vee},
\end{equation}
where $h^\vee$ is the dual Coxeter number of the Lie algebra $\g$.

Here we collect some basic facts about categories $\C(sl(2),\, k)$.

Simple objects $[j]$ of $\C(sl(2),\, k)$ are labelled by integers $j=0,...,k$. The decomposition of tensor products of simple objects  is given by
$$[i]\ot[j] = \left\{\begin{array}{cc}\bigoplus\displaylimits^{\min(i,j)}_{s=0} [i+j-2s], & i+j<k\\ \\ \bigoplus\displaylimits^{\min(i,j)}_{s=i+j-k} [i+j-2s], & i+j\geq k\end{array}\right.$$
There is a canonical ribbon structure on $\C(sl(2),\, k)$ :
$$\theta_{[j]} = e^{2\pi i \frac{j(j+2)}{4(k+2)}}\id_{[j]},$$
which gives the square of the braiding:
$$c_{[j],[t]}c_{[t],[j]} = \bigoplus_{s} e^{2\pi (h_{t+j-2s} - h_t-h_j)}\id_{[t+j-2s]},$$ where $h_j = \frac{j(j+2)}{4(k+2)}$.
The Frobenius-Perron dimensions of simple objects 
are
$$\FPdim([j]) = \frac{q^{j+1}-q^{-j-1}}{q-q^{-1}},\quad q = e^\frac{\pi i}{k+2}.$$
The  Frobenius-Perron dimension of of $\C(sl(2),\, k)$ is 
$$\FPdim(\C(sl(2),\, k)) = \frac{k+2}{2\sin^2(\frac{\pi}{k+2})}.$$
The multiplicative central charge of $\C(sl(2),\, k)$ is $$\xi(\C(sl(2),\, k)) = e^{2\pi i\frac{3k}{8(k+2)}}.$$
\end{section}


\begin{section}{\'Etale algebras in tensor products of braided fusion categories}
\label{Etale algebras in products}

Let $\C$ be a braided fusion category.

\subsection{\'Etale subalgebras}

\begin{lemma} \label{et0}
Let $A$ be an \'etale algebra in $\C$.
There is a bijection between \'etale algebras over $A$ and \'etale algebras in $\C_A^0$.
\end{lemma}
\begin{proof}
This statement was proved in \cite[Proposition 3.16]{DMNO}.
See also \cite[Lemma 4.13]{FFRS} and \cite[Proposition 2.3.3]{da}.
\end{proof}

\begin{lemma} 
\label{Lemma 3.2}
Let $A\in \C$ be a separable algebra and let $\D \subset \C$ be a fusion subcategory.
Then $A\cap \D$ is separable.
\end{lemma}
\begin{proof} Let us consider the category $\C_A$ as a module category over $\D$. It is clear
from definitions that $\underline{\Hom}_\D(A,A)=A\cap \D$, where $\underline{\Hom}_\D$
denotes the internal Hom in $\D$. 
Hence, the category $\D_{A\cap \D}$ coincides with $\D-$module subcategory of $\C_A$ generated by $A\in \C_A$, see \cite[Section 3.3]{O1}; in particular this category is semisimple. This completes the proof.
\end{proof}

\begin{corollary} \label{etsub}
Assume $\C$ is a braided fusion category. 
Let $A\in \C$ be an \'etale algebra and let $\D \subset \C$ be a fusion subcategory.
Then $A\cap \D$ is \'etale. $\square$
\end{corollary}

\begin{remark} An interesting open question is whether {\em any} subalgebra of an \'etale algebra is \'etale. 
\end{remark}

\subsection{Classification of \'etale algebras in tensor products}

Let $\C$ be a braided fusion category.  Let $G: \C \bt \C^{\rev} \to \Z(\C)$
be the canonical braided tensor functor, see \eqref{Gfun}. 
The tensor product functor $\ot: \C \bt \C^{\rev} \to \C$  factors through $G$ and so
has a natural structure of central functor. 
Let $I_\ot :\C \to \C \bt \C^{\rev}$ be the right adjoint functor of $\ot$. 
The object $I_\ot (\be)\in \C \bt \C^{\rev}$ has a structure of \'etale algebra, see Section~\ref{Prelim-etale}.

It is easy to compute that $I_\ot (\be)=\bigoplus_{X\in \O(\C)}X^*\bt X$. Thus, 
for a braided fusion category $\C$ the object $\bigoplus_{X\in \O(\C)}X^*\bt X\in \C \bt \C^{\rev}$ has
a canonical structure of connected \'etale algebra.

We now generalize this construction as follows. Let $\C$ and $\D$ be two braided fusion categories and
let $A_1\in \C$ and $A_2\in \D$ be two connected \'etale algebras. Let $\C_1 \subset \C^0_{A_1}$ and 
$\D_1\subset \D^0_{A_2}$ be two fusion subcategories, and let $\phi: \C_1^{\rev}\simeq \D_1$
be a braided equivalence. Using tensor functors 
\[
\C_1 \bt \C_1^{\rev}\xrightarrow{\id \bt \phi} \C_1 \bt \D_1\hookrightarrow
\C^0_{A_1}\bt \D^0_{A_2}=(\C \bt \D)^0_{A_1\bt A_2}
\] 
and Lemma \ref{et0} we can consider the algebra 
$I_\ot(\be)\in \C_1\bt \C_1^{\rev}$ as an \'etale algebra over $A_1\bt A_2$ in $\C \bt \D$. We will call this algebra 
$A(A_1, A_2, \C_1, \D_1, \phi)$.

\begin{lemma} 
\label{prdim}
We have 
\[
\FP(A(A_1, A_2, \C_1, \D_1,\phi))=\FP(A_1)\FP(A_2)\FP(\C_1).
\]
\end{lemma}
\begin{proof} For an object $M\in \C_{A_1}$ we have two possible Frobenius-Perron dimensions:
$\FP(M)$ where $M$ is considered as an object of fusion category $\C_{A_1}$ and $\FP_{\C}(M)$ where 
$M$ is considered as an object of fusion category $\C$. It is a straightforward consequence of
\cite[Proposition 8.7]{ENO1} that 
\[
\FP_{\C}(M)=\FP(M)\FP(A_1). 
\]
We have:
\begin{eqnarray*}
\FP(A(A_1, A_2, \C_1, \D_1,\phi))
&=& \sum_{X\in \O(\C_1)}\FP_\C(X^*)\FP_\D(\phi(X)) \\
&=& \sum_{X\in \O(\C_1)}\FP(A_1)\FP(X)^2\FP(A_2) \\
&=& \FP(A_1)\FP(A_2)\FP(\C_1).
\end{eqnarray*}
\end{proof}

Here is the main result of this section:

\begin{theorem} 
\label{etale in CD}
The algebras $A(A_1, A_2, \C_1, \D_1,\phi)$ are pairwise non-isomorphic 
and any connected \'etale algebra in $\C \bt \D$ is isomorphic to one of them.
\end{theorem}
\begin{proof} 
Let $A$ be an \'etale algebra in $\C \bt \D$. We will construct data 
$A_1, A_2, \C_1,\D_1,\phi$ such that $A\simeq A(A_1, A_2, \C_1, \D_1,\phi)$. 
Moreover, if the constructions below applied to $A=A(A_1, A_2, \C_1, \D_1,\phi)$ we recover the original $A_1, A_2, \C_1,\D_1,\phi$.

Set $A_1=A\cap (\C \bt \be)$ and $A_2= A\cap (\be \bt \D)$. 
By Corollary \ref{etsub} the algebras $A_1\in \C, A_2\in \D$ and $A_1\bt A_2\in \C \bt \D$ are connected \'etale. The algebra $A$ can be considered as an \'etale algebra in $(\C \bt \D)^0_{A_1\bt A_2}\simeq \C^0_{A_1}\bt \D^0_{A_2}$. Replacing $\C$ by $\C^0_{A_1}$ and $\D$ by 
$\D^0_{A_2}$ we reduce the Theorem to the case when $A_1=\be$ and $A_2=\be$.

Let $A\in \C \bt \D$ be an \'etale algebra such that $A\cap (\C \bt \be)=A\cap (\be \bt \D)=\be$. Consider
fusion category $\A =(\C \bt \D)_{A}$. The restrictions of the canonical central functor $\C \bt \D \to (\C \bt \D)_{A}$ to $\C =\C \bt \be$ and $\D =\be \bt \D$ are injective . Let $\A_1\subset \A$ be
the intersection of the images of $\C$ and $\D$ in $\A$. Then there are fusion subcategories $\C_1\subset \C$ and $\D_1 \subset \D$ such that the functors above restrict to equivalences $\C_1\simeq \A_1$ and $\D_1\simeq \A_1$. Combining these equivalences we get a tensor equivalence $\phi: \C_1\xrightarrow{\sim} \D_1$; it is clear that the algebra $A$ identifies with $\bigoplus_{X\in \O(\C_1)}X^*\bt \phi(X)$
(more precisely, it identifies
with the image of $I_\ot(\be)\in \C_1\bt \C_1^{\rev}$ under the equivalence $\Id \bt \phi$).

To finish the proof we need to show that the equivalence $\phi :\C_1 \xrightarrow{\sim} \D_1$ constructed above is 
{\em braided} when considered as a functor $\C_1^{\rev}\to \D_1$. For this we can assume that
that $\C=\C_1$ and $\D=\D_1$ (and so $\A=\A_1$). The functor $\C \bt \D \to \A$ is central, i.e.,  
it factorizes as  $\C \bt \D \to \Z(\A)\to \A$. The functors $\C \to \Z(\A)$ and $\D \to \Z(\A)$ are injective;
if we identify $\C$ and $\D$ with their images in $\Z(\A)$, then they centralize each other, that is
$\D \subset \C'_{\Z(\A)}$, where $\C'_{\Z(\A)}$ denotes the centralizer of $\C$ in $\Z(\A)$. 
Since $\FP(\Z(\A))=\FP(\A)^2$ and $\FP(\A)=\FP(\C)=\FP(\D)$, we see from 
by \eqref{prodFP} that $\FP(\C'_{\Z(\A)})=\FP(\D)$, therefore  $\D= \C'_{\Z(\A)}$. On the other hand, since $\C \to \Z(\A)\to \A$ is an equivalence,
we have $\C'_{\Z(\A)}\cong \C^{\rev}\subset \Z(\A)$. 
It follows from definitions that the functor $\phi$ is isomorphic to the composition of braided tensor functors $\C^{\rev} \xrightarrow{\sim}   \C'_{\Z(\A)} \xrightarrow{\sim}  \D$ and, hence, it is braided.
\end{proof}

\subsection{Lagrangian algebras in the center of non-degenerate braided fusion category}

Let $\C$ be a non-degenerate braided fusion category.  We introduced the notion of a Lagrangian algebra in $\C$
in Section~\ref{Prelim-etale}.

\begin{proposition} 
\label{parent}
Let $\C$ be a non-degenerate braided fusion category. 
A Lagrangian algebra in $\Z(\C)$ is of the form $A(A_1,A_2,\C_1,\D_1,\phi)$
with $\FP(A_1)=\FP(A_2)$, $\C_1=\C^0_{A_1}$, $\D_1=(\C^{\rev})^0_{A_2}$. 
\end{proposition}
\begin{proof} 
Since $\C$ is non-degenerate, we have $\Z(\C)\cong \C \bt \C^\rev$, see 
 Section~\ref{Prelim-brcat}. Using Lemma \ref{prdim} and  \eqref{FPdimCA0} we have  
\begin{eqnarray*}
\FP(A(A_1, A_2, \C_1, \D_1,\phi)) &\leq& 
\FP(A_1)\FP(A_2)\FP(\C^0_{A_1})  \\
&=& \frac{\FP(A_2)}{\FP(A_1)}\FP(\C)
\end{eqnarray*}
and, similarly,  
\[
\FP(A(A_1, A_2, \C_1, \D_1,\phi))\le \frac{\FP(A_1)}{\FP(A_2)}\FP(\C). 
\]
The result follows.
\end{proof}


\begin{corollary} 
Let $\C$ be a non-degenerate braided fusion category. Equivalence classes of indecomposable module
categories over $\C$ are parameterized by isomorphism classes of triples $(A_1,\, A_2,\,  \phi)$ where $A_1, A_2 \in \C$ are 
connected \'etale algebras and $\phi: \C^0_{A_1}\xrightarrow{\sim} (\C^0_{A_2})^\rev$ 
is a braided equivalence.
\end{corollary}
\begin{proof}
This statement follows by combining Proposition~\ref{parent} and the bijection between module categories
over a fusion category  and Lagrangian algebras in its center, see  Section~\ref{Prelim-etale}.
\end{proof}

\begin{remark} The case $A_1=A_2=\be$ corresponds to the {\em invertible} module categories
over $\C$, see \cite[\S 5.4]{ENO3}. The result above shows that the invertible module categories
over $\C$ are in bijection with braided autoequivalences of $\C$. This is a weak form
of \cite[Theorem 5.2]{ENO3}.
\end{remark}

\begin{remark} In the setup of Rational Conformal Field Theory, the category $\C$ is the representation
category of a vertex algebra and a Lagrangian algebra $A\in \C \bt \C^{\rev}$ (or rather the underlying
vector space) is the Hilbert space of physical states; the class $[A]\in K(\C \bt \C^{\rev})=K(\C)\ot_\mathbb{Z} K(\C)$ 
is the {\em modular invariant} (or partition function) of the theory. In the physical terminology, the modular invariant is of type I if $A_1=A_2$ and $\phi=\Id$; of type II if $A_1=A_2$ but $\phi \not \simeq \Id$; and is {\em heterotic} if $A_1\not \simeq A_2$ (one should be careful, sometimes the same terminology is used when $A_1, A_2, \phi$ are replaced by their classes in the Grothendieck group).
In this language, Proposition~\ref{parent} says that each modular invariant has two type I ``parents".
This result was initially observed by physicists \cite{MS, DV}; one finds mathematical treatments 
in \cite{BE, FFRS}.
\end{remark}

\begin{example} \label{cd0}
Let $\C$ and $\D$ be non-degenerate braided fusion categories. Recall from Section \ref{Prelim-etale}
that $\C \bt \D \simeq \Z(\A)$ if and only if $\C \bt \D$ contains a Lagrangian algebra $I(\be)$ (here $I: \A \to \Z(\A)$ is the right adjoint of the forgetful functor). Thus Theorem \ref{etale in CD}
implies that $\D^{\rev}\simeq \C_A^0$ for some connected \'etale algebra $A\in \C$ if and only if
there exists a braided equivalence $\C \bt \D \simeq \Z(\A)$ such that the forgetful functor $\D \to \A$
is injective. Moreover, in this case $\C \cap I(\be)=A$.
\end{example}

\subsection{Subcategories of $\Z(\A)$}
Let $\A$ be a fusion category and let $\Z(\A)$ be its Drinfeld center with forgetful functor $F:\Z(\A)\to \A$.
Let $\C \subset \Z(\A)$ be a fusion subcategory and let $\C'\subset \Z(\A)$ be its M\"uger centralizer in
$\Z(\A)$. 

\begin{theorem}
\label{insu}
 The forgetful functor $\C \to \A$ is injective 
 if and only if the forgetful functor $\C'\to \A$ is surjective.
\end{theorem}

The proof of Theorem~\ref{insu} is given in Section~\ref{Proof of insu}.

\begin{remark} Let $\C_1$ and $\D$ be non-degenerate braided fusion categories and let
$A\in \C_1\bt \D$ be a connected \'etale algebra such that $A\cap \C_1=\be$. Let 
$\C_2^{\rev}=(\C_1\bt \D)_A^0$. In view of Example \ref{cd0} the conditions above are equivalent
to the existence of a braided equivalence $\C_1\bt \D \bt \C_2\simeq \Z(\A)$ such that
\begin{enumerate}
\item[(1)] Forgetful functor $\C_1\to \A$ is injective and
\item[(2)] Forgetful functor $\C_1\bt \D \to \A$ is surjective.
\end{enumerate}
Since $\C_1'=\D \bt \C_2$ and $(\C_1\bt \D)'=\C_2$, Theorem \ref{insu} implies that the conditions above are equivalent to
\begin{enumerate}
\item[(1')] Forgetful functor $\D \bt \C_2 \to \A$ is surjective and
\item[(2')] Forgetful functor $\C_2\to \A$ is injective.
\end{enumerate}

In view of Example \ref{cd0} these conditions are equivalent to the existence of  an \'etale algebra
$B\in \D \bt \C_2$ such that $B\cap \C_2=\be$ and $(\D \bt \C_2)_B^0\simeq \C_1^{\rev}$.
Thus we proved the following result:

\begin{theorem} Let $\C_1, \C_2, \D$ be non-degenerate braided fusion categories. The existence of
\'etale algebra $A\in \C_1\bt \D$ such that $A\cap \C_1=\be$ and $(\C_1\bt \D)_A^0=\C_2^{\rev}$ is
equivalent to the existence of \'etale algebra
$B\in \D \bt \C_2$ such that $B\cap \C_2=\be$ and $(\D \bt \C_2)_B^0\simeq \C_1^{\rev}$. $\square$
\end{theorem}

This Theorem is a special case\footnote{In \cite[Theorem 7.20]{FFRS} the assumption on non-degeneracy of the categories $\C_1$ and $\C_2$ (but not $\D$!) is dropped.} of \cite[Theorem 7.20]{FFRS}; see {\em loc.\ cit.}\ for the explanation of its significance for the Rational Conformal Field Theory.
\end{remark}

\subsection{Proof of Theorem \ref{insu}} 
\label{Proof of insu}
Let $I: \A \to \Z(\A)$ be the right adjoint functor of the forgetful
functor $F:\Z(\A)\to \A$. Then $I(\be)\in \Z(\A)$ is an \'etale algebra as in  Section~\ref{Prelim-etale} 
and the functor  $I$ naturally upgrades to a tensor equivalence of $\A$ and the category $\Z(\A)_{I(\be)}$ of 
right $I(\be)-$modules, see \cite[Proposition 4.4]{DMNO}.
For any \'etale subalgebra $B\subset I(\be)$ let $\A(B)\subset \A$ consist of objects $X\in \A$ such that
$I(X)\in \Z(\A)_{I(\be)}$ is a dyslectic $B-$module. Then it is proved in \cite[Theorem 4.10]{DMNO} that the
assignment $B\mapsto \A(B)$ is an anti-isomorphism of the lattices of \'etale subalgebras of $I(\be)$ 
and of fusion subcategories of $\A$. In addition we have
\begin{equation} \label{fpab}
\FP(\A(B))=\frac{\FP(\A)}{\FP(B)}.
\end{equation}

Recall from Lemma~\ref{Lemma 3.2} that for a fusion subcategory $\C \subset \Z(\A)$, the subalgebra 
$\C \cap I(\be)\subset I(\be)$ is \'etale.

\begin{theorem}\label{better}
Let $\C \subset \Z(\A)$ be a fusion subcategory. Then $\A(\C \cap I(\be))$ is precisely
the image $F(\C')$ of $\C'$ in $\A$ under the forgetful functor.
\end{theorem}

\begin{proof} It is clear that the image of $\C'$ in $\A \simeq \Z(\A)_{I(\be)}$ consists of modules which
are dyslectic when restricted to $\C \cap I(\be)$, that is $F(\C')\subset \A(\C \cap I(\be))$. Hence
\begin{equation}\label{menshe}
\FP(F(\C'))\le \FP(\A(\C \cap I(\be))),
\end{equation}
with equality if and only if $F(\C')=\A(\C \cap I(\be))$.
Using \eqref{fpab} and \cite[Lemma 3.11]{DMNO} we see that \eqref{menshe} is equivalent to  
\begin{equation}\label{menshee}
\frac{\FP(\C')}{\FP(\C'\cap I(\be))}\le \frac{\FP(\A)}{\FP(\C \cap I(\be))}.
\end{equation}
Interchanging $\C$ and $\C'$ we get
\begin{equation}\label{mensheee}
\frac{\FP(\C)}{\FP(\C\cap I(\be))}\le \frac{\FP(\A)}{\FP(\C' \cap I(\be))}.
\end{equation}
Multiplying \eqref{menshee} and \eqref{mensheee} and canceling the denominators we get
$$\FP(\C)\FP(\C')\le \FP(\A)^2.$$
However it follows from  \eqref{prodFP} and \eqref{dimZ}
that  $\FP(\C)\FP(\C')=\FP(\Z(\A))=\FP(\A)^2.$ Thus we had an equality in \eqref{menshe} 
and Theorem is proved.
\end{proof}

{\em Proof of Theorem \ref{insu}.} The injectivity of the forgetful functor $F: \C \to \A$ is equivalent to
$\C \cap I(\be)=\be$. Since $B\mapsto \A(B)$ is an anti-isomorphism of lattices, 
the condition $\A(B)=\A$ 
is equivalent to $B=\be$. Thus Theorem \ref{insu} follows from Theorem \ref{better}. $\square$ 


\end{section}


\begin{section}{Tensor product decomposition  of a slightly degenerate braided fusion category}
\label{Tensor product decomposition}

\subsection{Non-degenerate braided fusion categories over a symmetric category}

\bde
Let $\E$ be a symmetric fusion category and let 
$\C$ be a braided fusion category over $\E$. 
We say that  $\C$ is {\em non-degenerate over $\E$} if $\E=\C'$, i.e., 
if $\E$  coincides with the symmetric center of $\C$. 
\ede

\begin{remark}
In this terminology a slightly degenerate braided fusion category is a non-degenerate braided
fusion category over $\sVec$.
\end{remark}

The following is a generalization of the decomposition theorem (\cite{Mu2},
\cite[Theorem 3.13]{DGNO1}) for non-degenerate braided fusion categories 
to the case of categories over $\E$. 

\begin{proposition}
\label{E Muger thm}
Let $\C$ be a non-degenerate braided fusion category over $\E$ and let $\D \subset \C$
be its $\E$-subcategory.  Suppose that $\D$ is non-degenerate over $\E$.  Then the centralizer $\D'$ of $\D$ in $\C$ is also non-degenerate over $\E$
and there is a  braided tensor equivalence over $\E$: 
\begin{equation}
\label{E Muger}
\C \cong \D \bte \D'.
\end{equation}
\end{proposition}
\begin{proof}
Since $\D$ is non-degenerate over $\E$, we have $\D \cap \D' \cong \E$.  By  \eqref{D''}  
we have $\D'' = \D \vee \C' =\D$. Hence, $\D'\cap \D'' = \D'\cap \D$ and $\D'$ is slightly degenerate.  

To establish equivalence \eqref{E Muger} first note that $\D \vee \D' =\C$. Indeed,   $\D \vee \D'$
is a fusion subcategory of $\C$ and  
\[
\FPdim(\D \vee \D') =  \frac{\FPdim(\D)\FPdim(\D')}{\FPdim(\E)} =\FPdim(\C)
\]
by \eqref{prodFP} and \cite[Lemma 3.38]{DGNO1}. The tensor multiplication defines a surjective
braided tensor functor 
\[
\ot: \D \bt \D' \to \D \vee \D' =\C.
\]  
Let $R: \C \to \D \bt \D' $ be the right adjoint of  $\ot$.  Since $R(\E)\subset \E \bt \E$
we conclude that $ \D \vee \D'  \cong (\D \bt \D')_{R(\be)} \cong \D \bte \D'$,
see Section~\ref{Prelim-etale} and Remark~\ref{over E = Amod}.
\end{proof}

Let $\E$ be a symmetric fusion category and let $\A$ be a fusion category
over $\E$ such that the composition of $\E\to \Z(\A)$ and the forgetful functor
$\Z(\A) \to \A$ is fully faithful. We will  denote  by $\Z(\A,\, \E)$ the 
centralizer of $\E$ in $\Z(\A)$.   Observe that $\Z(\A,\, \E)$ is a non-degenerate
braided fusion $\E$-category.
Let $F_\E : \Z(\A,\, \E) \to \A$ denote the forgetful
functor.  By  Theorem~\ref{insu}  the functor $F_\E$ is surjective.

\begin{corollary}
\label{e Muger}
Let $\C$ be a non-degenerate braided fusion category over $\E$. Then there is a  braided tensor equivalence over $\E$
\begin{equation}
\label{e center of sd}
\Z(\C,\, \E) \cong \C \bte\C^{rev}. 
\end{equation}
\end{corollary}
\begin{proof}
Let us view $\C$ and $\C^\rev$ as fusion subcategories of $\Z(\C)$.
Clearly, they are centralizers of each other. 
By definition, $\Z(\C,\, \E)$ is the centralizer of $\E$ in $\Z(\C)$, therefore,
using \cite[Lemma 2.8]{Mu2} and \eqref{D''} we get
\[
\Z(\C,\, \E) = (\C \cap \C^\rev)' = \C \vee \C^{rev},
\] 
and  the result follows from Proposition~\ref{E Muger thm}.
\end{proof}

Let $A$ be a connected \'etale algebra in $\Z(\A,\, \E)$. The category 
$\Rep_\A(A)$ of $F_\E(A)$-modules in  $\A$  has a canonical structure
of a {\em multi-fusion} category over $\E$ (i.e., its unit object 
is not necessarily simple) with the embedding 
$\E \hookrightarrow  \Z(\Rep_\A(A))$
given by the free module functor, i.e., $X \mapsto X \ot F_\E(A)$
for all $X \in \E$.  

\begin{proposition}
\label{ZAT0}
There is a canonical braided tensor equivalence
\begin{equation}
\label{318}
\Z(\A,\, \E)_A^0\cong  \Z(\Rep_\A(A),\, \E).
\end{equation}
\end{proposition}
\begin{proof}
There is a  braided equivalence $\Z(\A)_A^0 \cong  \Z(\Rep_\A(A))$, see \cite[Corollary 4.5]{Sch} and
\cite[Theorem 3.20]{DMNO}. It  commutes
with inclusions of  $\E$  and, hence, restricts to
a braided equivalence \eqref{318}.
\end{proof}

Let $A$ be a connected \'etale algebra in a non-degenerate braided fusion category $\C$ over $\E$
such that $A\cap \E=\be$.  This algebra $A$  can be considered as a connected \'etale algebra in $\C^{\rev}$
and in $\Z(\C,\, \E)$ via the embedding 
\[
\C^{\rev}= \E \bte \C^{\rev}\hookrightarrow \C \bte \C^{\rev}
\cong \Z(\C,\, \E). 
\]

Under equivalence \eqref{e center of sd} we have
\begin{equation}
\label{Vitia's gift to humanity}
\Z(\C,\, \E)_A \cong \C \bte \C^{\rev}_A \quad \mbox{and} \quad
\Z(\C,\, \E)_A^0 \cong \C \bte (\C_A^0)^{\rev}.
\end{equation}

\begin{corollary} 
\label{ZCA}
Let $\C$ be a non-degenerate braided fusion category and let $A\in \C$
be a connected \'etale algebra such that $A\cap \E =\be$. 
There is a braided equivalence 
\[
\Z(\C_A,\, \E)\cong  \C \bt_\E (\C_A^0)^{\rev}.
\]
In particular, the category $\C_A^0$ is  non-degenerate over $\E$.
\end{corollary}
\begin{proof}
This is a direct consequence of equivalence \eqref{Vitia's gift to humanity}
and  Proposition~\ref{ZAT0}.
\end{proof}


\begin{proposition}
\label{e transitivity}
Let $\A_1,\, \A_2$ be  fusion categories over $\E$. There is an equivalence of braided fusion categories over $\E$
\begin{equation}
\label{Alexei's equivalence+}
\Z(\A_1,\, \E) \bte \Z(\A_2,\, \E) \cong \Z(\A_1 \bte \A_2,\, \E).
\end{equation}
\end{proposition}
\begin{proof}
Start with $\E \bt \E \hookrightarrow \Z(\A_1\bt \A_2)$ and let $\F$ 
denote the Tannakian subcategory of $\E \bt \E$ such that 
$\A_1 \bte\A_2 = (\A_1\bt \A_2) \bt_\F \Vec$, see Remark~\ref{tensor over E = equiv}.

Consider the centralizers of $\E \bt \E$ and $\F$  in $\Z(\A_1\bt \A_2)$.
We have embedding $\Z(\A_1\bt \A_2,\E \bt \E)\subset \Z(\A_1\bt \A_2,\F)$.

By Proposition~\ref{Center of the de-equivariantization}   braided fusion category
$\Z(\A_1 \bte \A_2)$ is equivalent to  the de-equivariantization
$\Z(\A_1\bt \A_2,\F)\bt_\F \Vec$.  The category $\Z(\A_1 \bte \A_2,\, \E)$ therefore is
identified with the centralizer of $(\E \bt \E) \bt_\F \Vec \cong \E$ in
$\Z(\A_1\bt \A_2,\F) \bt_\F \Vec$.

On the other hand, $\Z(\A_1,\, \E) \bt \Z(\A_2,\, \E)  = \Z(\A_1\bt \A_2,\, \E\bt \E) $  
and,  therefore,  $\Z(\A_1,\, \E) \bte \Z(\A_2,\, \E)$
is identified with $\Z(\A_1\bt \A_2,\E \bt \E) \bt_\F \Vec$.

Thus, existence of the braided equivalence \eqref{Alexei's equivalence+}
follows from Proposition~\ref{De-equivariantization commutes with centralizers}.
\end{proof}

\subsection{Slightly degenerate braided fusion categories}

\begin{definition}  
\label{s-category}
A braided  {\em s-category} is a braided fusion category over $\sVec$.
\end{definition}

In other words,  a braided s-category is a braided fusion category $\C$
equipped with a braided tensor functor $T: \sVec \to \C'$.  Such
a functor $T$ is necessarily injective.

An {\em s-subcategory} of a braided  s-category $\C$ is a fusion subcategory 
$\C_1\subset \C$ containing $T(\sVec)$. Note that an s-subcategory of a braided  s-category is a fusion category over $\sVec$.

\begin{definition}
\label{s-simple}
A braided s-category $\C$ is called {\em s-simple} if it is non-pointed 
and  has no s-subcategories except $\sVec$ and $\C$.
\end{definition}

\begin{remark}
An $s$-simple category is slightly degenerate (see Definition~\ref{def sldeg}).
\end{remark}

Let $\C_1,\,\dots,\,\C_n$ be braided s-categories. One defines
their tensor product over $\sVec$ by iterating  the construction of Section~\ref{Prelim-over E}. 
The equivalence class of the resulting braided s-category does not depend on the order 
in which products are taken.  By transitivity of de-equivariantization (see \cite[Lemma 4.32]{DGNO1})
there is a braided tensor equivalence
\[
\C_1 \bt_{\sVec} \cdots \bt_{\sVec} \C_n \cong (\C_1 \bt \cdots \bt \C_n)  \bt_{\tilde{\E}} \Vec,
\]
where $\tilde{\E}$ is the maximal Tannakian subcategory of  $(\sVec)^{\bt n} \subset \C_1 \bt \cdots \bt \C_n$.

Clearly, $\tilde{\E} \cong \Rep(G)$ where $G \cong (\mathbb{Z}/
2\mathbb{Z})^{n-1}$ is an elementary Abelian $2$-group. 
There is canonical action of the group $G$ on
the tensor product 
$\C_1\bts  \cdots \bts \C_n$ by braided tensor autoequivalences, see \cite[Section 4]{DGNO1}.

\begin{proposition}
\label{action identical}
Let $\C_1,\dots, \C_n$ be  slightly degenerate braided fusion categories.
The canonical action of  $G$
on $\C_1\bts  \cdots \bts \C_n$ is identical on objects, i.e., $g(X)\cong X$
for all objects $X\in \C_1\bts  \cdots \bts \C_n$ and all $g \in G$.
\end{proposition}
\begin{proof}
Let $A=\mbox{Fun}(G)$ be the regular algebra in
$\tilde{\E}$.
Recall that  the action of $G$ on $\C_1\bts\,  \cdots \,\bts \C_n
= (\C_1\bt \,\cdots \,\ \bt \C_n)_A$ is given by the action of $G$
on $A$ by translations.  In particular, $g(X)\cong X$ for every free $A$-module $X = A \ot Y,\, Y \in 
\C_1\bt \, \cdots \,\bt \C_n$ and $g \in G$.  Thus, it suffices to show
that every simple $A$-module in $\C_1\bts \, \cdots \,\bts \C_n$ is free.
Since the functor of taking free $A$-module 
\[
(\C_1\bt\, \cdots \,\bt \C_n) \to (\C_1\bt \,\cdots \,\bt \C_n)_A : Y \mapsto  Y \ot A
\]
is surjective, the last property is  equivalent to every free $A$-module being simple. 


In our situation  $A$ is equal to the direct sum
of objects $\delta_1 \bt\, \cdots \, \bt \delta_n$
in $\C_1\bt \,\cdots\, \bt \C_n$, where each $\delta_i,\, i=1,\dots,n$,
is isomorphic to $\be$ or $\delta$ and the number of $i's$ for which $\delta_i \cong \delta$
is even.  Since $\delta \ot Y \not\cong Y$  for any simple object $Y$ of a slightly
degenerate category (see Section~\ref{Prelim-brcat}) 
it follows that every free  $A$-module in $\C_1\bt\, \cdots\, \bt \C_n$
is a direct sum of $2^{n-1}$ non-isomorphic simple objects in $\C_1\bt \,\cdots\, \bt \C_n$.
It is easy to see that every such $A$-module is necessarily simple.
\end{proof}

\begin{corollary}
\label{lattice iso}
Let $\C_1,\,\dots,\, \C_n$ be slightly degenerate braided fusion categories.
Let $\tilde{\E}$ denote the maximal Tannakian subcategory
of $(\sVec)^{\bt n} \hookrightarrow \C_1\bt\, \cdots \,\bt \C_n$. 
There is an isomorphism between 
the lattice of fusion subcategories of $\C_1\bt \,\cdots \bt \,\C_n$ containing $\tilde{\E}$ and 
the lattice of fusion subcategories of   $\C_1\bts \, \cdots \, \bts \C_n$
given by
\[
\D \mapsto \D \bt_{\tilde{\E}} \Vec. 
\]
\end{corollary}
\begin{proof}
By  \cite[Proposition 4.30(i)]{DGNO1} there is an isomorphism between the lattice 
of fusion subcategories of $\C_1\bt \,\cdots \bt \,\C_n$ containing $\tilde{\E}$
and  the lattice of $G$-stable fusion subcategories of   $\C_1\bts \, \cdots \, \bts \C_n
= (\C_1\bt \,\cdots \bt \,\C_n) \bt_{\tilde{\E}} \Vec$ given by $\D \mapsto \D \bt_{\tilde{\E}} \Vec$.
So the statement follows from Proposition~\ref{action identical}.
\end{proof}

\subsection{Decomposition Theorem}

The next Theorem is a special case of a  result from \cite{DGNO2}.
We include the proof for the reader's convenience.

\begin{theorem}
\label{decomposition theorem} 
Let $\C$ be a slightly degenerate braided fusion category.
Suppose $\C$ has no Tannakian
subcategories other than $\Vec$. Then
\begin{enumerate}
\item[(1)] There exist s-simple subcategories $\C_1,\dots, \C_m \subset \C$ 
determined uniquely up to a permutation of indices such that
\begin{equation}
\label{decomposition S}
\C \cong  \C_{pt} \bts \C_1 \bts \cdots \bts \C_m.
\end{equation}
\item[(2)] Every fusion subcategory $\D \subset \C$ has the form
\[
\D =  \D_{pt} \bts \C_{i_1} \bts \cdots \bts \C_{i_k}
\]
for a subset $\{i_1,\dots, i_k\} \subset \{1,\dots,m\}$.
\end{enumerate}
\end{theorem}
\begin{proof}
Since $\C$ has no non-trivial Tannakian subcategories, we see 
that $\Vec$ and $\C'=\sVec$ are the only symmetric subcategories of $\C$.

For any fusion subcategory  $\D\subset \C$ the category $\D \cap \D'$ is
symmetric.  Therefore, $\D$ is  either non-degenerate
(if $\sVec\not\subset \D$) or slightly degenerate (if $\sVec\subset \D$).

In particular, $\C_{pt}$ is slightly degenerate. Therefore, $\C_{pt}'$
is slightly degenerate  and  $\C \cong \C_{pt}  \bts \C_{pt}'$
by   Proposition~\ref{E Muger thm}. So it suffices to prove the
Theorem in the case when $\C_{pt}=\sVec$. 

If $\C$ is s-simple then there is nothing to prove.  Otherwise,
let $\C_1$ be an s-simple subcategory of $\C$. Then
$\C \cong \C_1 \bts \C_1'$ by   Proposition~\ref{E Muger thm}. 
If $\C_1'$ is not s-simple choose its s-simple subcategory
$\C_2$ and apply Proposition~\ref{E Muger thm} again.
Continuing in this way we obtain decomposition \eqref{decomposition S}.

Let us prove the uniqueness assertion. Let $\D$ be 
an s-subcategory of $\C$. 
By Corollary~\ref{lattice iso} there is a fusion subcategory
$\tilde{\D} \subset  \C_1\bt \,\cdots\, \bt  \C_n$
containing  $(\sVec)^{\bt n}$ such that 
\[
\D \cong \tilde{\D} \bt_{\tilde{\E}} \Vec.
\]
Let $\tilde{\D} _i \subset \C_i$
be the fusion subcategory consisting of all objects $X_i\in \C_i$ such that 
there is an object  $X =X_1\bt \cdots \bt X_i \bt \cdots \bt X_n \in \tilde{\D}$. 
Since $\C_i$ is s-simple we have either $\tilde{\D_i} = \sVec$ or
$\tilde{\D} _i = \C_i$.  Let us analyze  the former case. 
Let   $\B_i \subset \C_i$ be a fusion subcategory generated by objects contained in $X_i \ot X_i^*$,
where $X_i$ is simple object in $\tilde{\D} _i$.
Since 
\begin{equation*}
X\ot X^*=(X_1\ot X_1^*)\boxtimes\cdots\boxtimes(X_n\ot X_n^*)\succ\be\boxtimes\cdots\boxtimes\be\boxtimes(X_i\otimes X_i^*)\boxtimes\be\boxtimes\cdots\boxtimes\be,
\end{equation*}
the category $\B_i$  is contained in $\tilde{\D}$. 
We claim that $\B_i =\C_i$. Indeed, otherwise $\B_i =\sVec$ since $\C_i$ is s-simple. Hence, $\FPdim(Y) =\sqrt{2}$
for every simple non-invertible object $Y\in \tilde{\D} _i$.   
Therefore, $Y\ot Y =\be \oplus \delta$ and so $Y$ generates 
an Ising subcategory $\mathcal{I} \subset \tilde{\D}_i$.
But this is impossible since on the one hand
$\mathcal{I}$ is non-degenerate  (see Section~\ref{Prelim-brcat})
and on the other hand $\tilde{\D} _i' =\sVec \subset \mathcal{I} \cap \mathcal{I}'$.
We conclude that either $\tilde{\D}_i = \sVec$  or $\C_i \subset \D$.
Therefore, 
\[
\tilde{\D} = (\sVec)^{\bt n} \vee (  \C_{i_1} \bt\, \cdots \, \bt \C_{i_k}), 
\]
where 
$\{ i_1,\dots, i_k \}$ is a subset of  $\{ 1,2,\dots, n\}$.
It follows that 
\[
\D \cong \C_{i_1} \bts\, \cdots \, \bts \C_{i_k}.
\]
In particular, if $\D$
is s-simple then $\D = \C_i$ for some $i$. This implies 
uniqueness of the tensor decomposition of $\C$.
\end{proof} 

\end{section}

\begin{section}{The Witt group of slightly degenerate braided fusion categories}
\label{Witt section}

\subsection{The Witt group of non-degenerate  braided fusion categories over $\E$}
\label{E-Witt}

\begin{definition} 
\label{eWitt eq}
Let $\C_1$ and $\C_2$ be
non-degenerate braided fusion categories over a symmetric fusion category $\E$
such that the corresponding braided tensor fumctors $\E \to \C_i,\,i=1,2,$ are fully faithful.
We will say that $\C_1$ and $\C_2$  are {\em Witt equivalent} if there exist  
fusion categories $\A_1$, $\A_2$ over $\E$ and  a braided equivalence over $\E$
\[
\C_1\bte \Z(\A_1,\, \E)\cong \C_2\bte \Z(\A_2,\, \E).
\] 
\end{definition}

It follows from Proposition~\ref{e transitivity} that  the Witt equivalence of non-degenerate braided fusion categories over $\E$
is indeed an equivalence relation. We will denote the Witt equivalence class of a non-degenerate braided fusion  $\E$-category 
$\C$ by $[\C]$. The set of Witt equivalence classes of slightly degenerate braided fusion  categories
will be denoted $\W(\E)$. Clearly, $\W(\E)$ is a commutative monoid with respect to the
multiplication  $\bte$. The unit of this monoid is $[\E]$.

\begin{lemma} 
The monoid $\W(\E)$ is a group.
\end{lemma}
\begin{proof} 
This immediately follows from Corollary~\ref{e Muger}.
\end{proof}

\begin{proposition}
Let $A$ be a connected \'etale algebra in a non-degenerate category $\C$ over $\E$ such that 
$A\cap \E=\be$. Then $[\C] =[\C_A^0]$ in $\W(\E)$. 
\end{proposition}
\begin{proof}
The result  follows from Corollary~\ref{ZCA}.
\end{proof}

\begin{definition}
Let $\E$ be a symmetric fusion category and let $\C$ be a braided fusion category
over $\E$. We say that $\C$ is {\em completely $\E$-anisotropic} if every connected \'etale algebra
in $\C$ belongs to $\E$.
\end{definition}

For  $\E=\Vec,\, \sVec$ the above notion coincides with the usual complete anisotropy, cf.\
Section~\ref{Prelim-etale}.

\begin{theorem}
\label{5.12}
Each equivalence class in $\W(\E)$ contains a completely $\E$-anisotropic  category
category that is unique up to a braided equivalence. 
\end{theorem}
\begin{proof}
This is completely parallel to \cite[Theorem 5.13]{DMNO}. 
\end{proof}

\begin{proposition} 
\label{whenCisinE}
Let $\C$ be a non-degenerate braided fusion category over $\E$.
Then $\C \in [\E]$ if and only if there exist a fusion
category $\A$ over $\E$  
and an equivalence $\C \cong \Z(\A,\, \E)$ of braided fusion categories over $\E$.
\end{proposition}
\begin{proof}
This is completely parallel to \cite[Proposition 5.8]{DMNO}. 
By definition, $\C \in [\E]$ if and only if 
\[
\C \bte \Z(\A_1,\, \E)  \cong \Z(\A_2,\, \E) .
\] 
Let $A_1\in \Z(\A_1,\, \E)$ be the connected \'etale algebra 
corresponding to the forgetful central functor $\Z(\A_1,\, \E)\to \A_1$.
We have
\[
\C \cong  \Z(\A_2,\, \E)_{A_1}^0 \cong \Z( \Rep_{\A_2}(A_1),\, \E),
\]
where the last equivalence is by Proposition~\ref{ZAT0}.
\end{proof}

We have $\W(\Vec)=\W$, the Witt group of non-degenerate braided fusion categories
introduced in \cite{DMNO}, see Section~\ref{Prelim-Witt}. 

Let us denote $s\W:= \W(\sVec)$.

\begin{definition}
The group $s\W$ is called the  {\em Witt group of slightly degenerate braided fusion categories}. 
\end{definition}
 
The group $s\W$ is the main object of our  analysis in the remainder of this paper. 

\begin{remark}
The base change along $\E\to\F$ defines a group homomorphism $\W(\E)\to \W(\F)$. 
Since by  Deligne's theorem \cite{D2} any symmetric fusion category $\E$ has a surjective braided 
tensor functor either into $\Vec$ or to $\sVec$ (i.e., $\E$ is either Tannakian or super-Tannakian)
it follows that  each $\W(\E)$ has a canonical homomorphism either to  $\W$ or to $s\W$.
\end{remark}

\subsection{Relations in $\sW$}

Let $\C$ be a completely anisotropic slightly degenerate braided fusion category and let 
\begin{equation}
\label{decomposition S'}
\C \cong  \C_{pt} \bts \C_1 \bts \cdots \bts \C_m, 
\end{equation}
be its decomposition into tensor product of s-simple subcategories from Theorem~\ref{decomposition theorem}.
Then categories $\C_{pt},\, \C_1,\dots,\C_m$  are completely anisotropic. 

Next Theorem extends  \cite[Theorem 5.19]{DMNO}
to slightly degenerate categories. It describes all central functors
on tensor products  of completely anisotropic s-simple
braided fusion categories (and, hence, all connected \'etale algebras in such products).

\begin{theorem}
\label{cenfunslideg}
Let $\C_1,\dots, \C_m$ be completely anisotropic s-simple
braided fusion categories.
Let 
\[
F:  \C_1 \bts\, \cdots \, \bts \C_m \to \A
\]
be a surjective central tensor functor.
\begin{enumerate}
\item[(1)] There is a subset $J \subset \{1,\dots,m\}$  such that
\[
\A \cong \underset{j\in J}{\bts}\, \C_j.
\]
\item[(2)] There is a surjective map $f:\{1,\dots,m\}\to J$ such that $|f^{-1}(j)| \leq 2$ for all $j\in J$ 
and 
\[
F=\underset{j\in J}{\bts} F_j, \quad \mbox{where} \quad
F_j :  \underset{i \in f^{-1}(j)}{\bts}\, \C_i \to \C_j \quad
\mbox{ is the restriction of } F.
\]
\item[(3)]
If $f^{-1}(j) =\{i\}$ then $F_j: \C_i \xrightarrow{\sim} \C_j$ is tensor  equivalence.  If $f^{-1}(j) =\{i,\, i'\}$  then 
there is a braided tensor equivalence $\C_{i'}\cong \C_i^\rev$ and $F_j$
is the forgetful tensor  functor
\begin{equation}
\label{Fj}
\C_i \bts \C_{i'} \cong  \C_j \bts \C_{j}^\rev \cong \Z(\C_j,\, \sVec) \xrightarrow{\text{Forget}} \C_{j} \cong \C_i. 
\end{equation}
\end{enumerate}
\end{theorem}
\begin{proof}
Let $\A_j\subset \A$
denote the image of $\C_i$ in $\A$, $j=1,\dots,m$. Since $\C_j$
is completely anisotropic, we have $\A_j \cong \C_j$. 
Clearly, $\A =\vee_{j=1}^m\, \A_j$.

The set $J$ is formed
by induction as follows. Set $J_1=\{1\}$.  Suppose for some $k<m$ a subset $J_k$
is chosen in such way that  
\[
\vee_{j=1}^k\, \A_j \cong \underset{j\in J_k}{\bts}\, \C_j.
\]  
Since $\A_{k+1}\cong \C_{k+1}$ is s-simple we have  either $\A_{k+1} \subset \vee_{j=1}^k\, \A_i$ 
or $\A_{k+1} \cap \left( \vee_{j=1}^k\, \A_i\right) 
=\sVec$. In the former case set $J_{k+1} =J_k$. Clearly, 
we  have 
\[
\vee_{j=1}^{k+1}\, \A_j \cong \underset{j\in J_{k+1}}{\bts}\, \C_j.
\]
In the latter
case  $\vee_{j=1}^{k+1}\, \A_j$ is generated by two  fusion subcategories,
$\vee_{j=1}^{k}\, \A_j$ and $\A_{k+1}$, whose intersection is $\sVec$.
The composition of $F$ with the tensor product of $\C$
gives rise to a surjective tensor functor
\begin{equation}
\label{Vitia's step}
\left(  \underset{j\in J_k}{\bts}\, \C_j   \right) \bts \C_{k+1} \to \vee_{j=1}^{k+1}\, \A_j.
\end{equation}
It follows from \cite[Lemma 3.38] {DGNO1} that both sides of \eqref{Vitia's step}
have equal Frobenius-Perron dimension. Therefore, \eqref{Vitia's step} 
is an equivalence and $J:= J_n$ is the required subset. This proves  part (1).

The functor $F$ factors as
\[
 \C_1 \bts\, \cdots \, \bts \C_m \xrightarrow{F'} \Z(\A) \cong \underset{j\in J}{\bts}\, (\C_j \bts \C_j^\rev)
 \to  \A \cong \underset{j\in J}{\bts}\, \C_j.
\]
Since $\C_1 \bts\, \cdots \, \bts \C_n$ is slightly degenerate, 
the braided tensor functor $F'$ is  injective.  Define an embedding
$\{1,\dots,m\} \hookrightarrow J \sqcup J$ by $f(i)=j \in J\sqcup \emptyset$
if $F'(\C_i)= \C_j$  and $f(i)=j \in \emptyset \sqcup J$
if $F'(\C_i)= \C_j^\rev$. This embedding is well defined by  
Theorem~\ref{decomposition theorem}.  Let
\[
f: \{1,\dots,m\} \hookrightarrow J \sqcup J \to J
\]
to be the composition of the above embedding and
the diagonal map. Since $F$ is surjective, so is $f$.
Since $\C_1 \bts\, \cdots \, \bts \C_m$ is slightly degenerate
the braided tensor functor $F'$ is injective.  Therefore,
$f^{-1}(j)$ contains at most $2$ elements for every $j\in J$ and
$F$ decomposes as a product  of central tensor functors 
\[
F_j :  \underset{i \in f^{-1}(j)}{\bts}\, \C_i \to \C_j.
\]
This proves part (2).  Part (3) is an immediate consequence of the above 
definition of $f$.
\end{proof}

%
%

\begin{corollary}
\label{relations in sW}
Let $\C_1,\dots,\C_m$ be completely anisotropic s-simple braided categories.
Suppose that there exists a fusion category $\A$ over $\sVec$ such that
\[
\C_1\bts \,\cdots\, \bts  \C_m \cong \Z(\A,\, \sVec).
\]
Then there exists a fixed point free  involution  $a$ of the set $\{1,\dots, m\}$ such that  $\C_i \cong \C_{a(i)}^{\rev}$
for all $i= 1,\dots, m$. 
\end{corollary}
\begin{proof}
Let us use the notation of Theorem~\ref{cenfunslideg}.
In this case $\{1,\dots, m\} = J \sqcup J$ and the map $f: \{1,\dots, m\} \to J$ is two-to-one.
Hence, it gives rise to a fixed point free involution of  $\{1,\dots, m\}$.
\end{proof}

\begin{remark}
Corollary~\ref{relations in sW} means that there are no non-trivial relations between
the Witt classes of completely anisotropic  s-simple braided categories in $s\W$
except possibly relations of the form $[\C]=[\C]^{-1}$.
\end{remark}

\begin{corollary}
\label{slight property S}
Let $\C$ be a completely anisotropic slightly degenerate
braided fusion category. Suppose
that $\C_{pt} = \sVec$.  If $[\C]\neq [\sVec]$ 
then the order of $[\C]$ in $\sW$ is either $2$ or $\infty$. 
\end{corollary}
\begin{proof}
By Theorem~\ref{decomposition theorem} we have
$\C\cong \C_1\bts \,\cdots\, \bts  \C_k $, where $\C_1,\dots,\C_k$
are completely anisotropic s-simple categories. 

By Proposition~\ref{whenCisinE} the order of $[\C]$ is finite if and only if
$\C^{\bt n} \cong \Z(\A,\, \sVec)$ for some fusion category $\A$ over $\sVec$.
It follows from Corollary~\ref{relations in sW} that there is an involution
$a$ of $\{1,\dots,k\}$ (not necessarily fixed point free) such that
$\C_i \cong \C_{a(i)}^\rev$ for each $i=1,\dots,k$.  Thus, after cancellations
$[\C]$ is equal to the product of classes $[\C_i]$ with $\C_i \cong \C_i^\rev$,
i.e., of elements of order $2$.
\end{proof}

\subsection{A canonical homomorphism $\W \to s\W$}
\label{homS}

Let $\W$ be the Witt group of non-degenerate braided fusion categories, see
Section~\ref{Prelim-Witt} and  \cite{DMNO}.
Define a map 
\begin{equation}
\label{map S}
S: \W \to \sW : [\C] \mapsto [\C \bt \sVec].
\end{equation}

\begin{proposition}
\label{S well defined}
The map \eqref{map S} is a well defined homomorphism.
\end{proposition}
\begin{proof}
It suffices to check that (a) $S$ maps the trivial class in $\W$ to the trivial class in $s\W$
and that (b) $S$ is multiplicative. Let $\C$ be a non-degenerate
braided fusion category such that  $\C \cong \Z(\A)$ for some fusion category $\A$.
Let us view $\sVec$ as a symmetric
subcategory of $\Z(\Vec_{\mathbb{Z}/2\mathbb{Z}})$. Then $\C \bt \sVec$  is the centralizer
of $\Vec \bt \sVec$ in $\C \bt \Z(\Vec_{\mathbb{Z}/2\mathbb{Z}})\cong \Z(\A \bt \Vec_{\mathbb{Z}/2\mathbb{Z}})$,
i.e., the class of $S([\C])=[\C \bt \sVec]$ in $s\W$ is trivial. This proves (a).  The verification
of (b) is straightforward:
\[
S([\C_1]) S([\C_2]) = [(\C_1 \bt \sVec) \bts (\C_2 \bt \sVec) ] = [\C_1 \bt \C_2 \bt \sVec ] = S([\C_1][\C_2]),
\]
for all non-degenerate braided fusion categories $\C_1$ and $\C_2$.
\end{proof}

Recall from Section~\ref{Prelim-Witt} that $\W$ contains a cyclic subgroup $\W_{Ising}$ of order $16$
generated by the Witt classes of Ising braided fusion categories. 

\begin{proposition}
\label{kernel of S}
The kernel of homomorphism \eqref{map S} is $\W_{Ising}$.
\end{proposition} 
\begin{proof}
Let $\C$ be a non-degenerate
braided fusion category such that
\[
\C \bt \sVec \cong \Z(\A,\, \sVec)
\]
for some fusion category $\A$ over $\sVec$.  
Then $\Z(\A)\cong \C \bt \C_1$ where $\C_1$ is a non-degenerate braided
fusion category of the Frobenius-Perron dimension $4$ containing $\sVec$.  It is easy to
see that categories with the above property are precisely Ising braided categories and pointed
braided categories $\C(A, \, q)$ associated with metric groups $(A,q)$ of order 4 such that there
exists $u\in A$ with $q(u)=-1$.
Since $[\C]=[\C_1]^{-1}$ in $\W$ we see that the kernel of $S$
is $\W_{Ising}$ (see Section \ref{Prelim-Witt}).
\end{proof}

\begin{question}
At the moment of writing we do not know whether the homomorphism \eqref{map S} is surjective.
More generally, can one always embed a slightly degenerate  braided fusion category $\C$ into
a non-degenerate braide fusion category $\D$ such that $\FPdim(\D)=2 \FPdim(\C)$? 
\end{question}

\subsection{The structure of $\sW$}
\label{structure}

Let $s\W_{pt}$ denote the subgroup of $s\W$
generated by the Witt classes of  slightly degenerate pointed braided fusion categories.

\begin{proposition}
\label{sWpt}
We have  
\[
s\W_{pt} = \bigoplus_{p \text{ is prime}}\, s\W_{pt}(p),
\] 
where $s\W_{pt}(p)$ is the group generated by classes of slightly degenerate pre-metric
$p$-groups:
\[
s\W_{pt}(p) \cong 
\begin{cases}
\mathbb{Z}/2\mathbb{Z} & \text{if } p=2, \\
\mathbb{Z}/2\mathbb{Z} \oplus \mathbb{Z}/2\mathbb{Z} & \text{if } p\equiv 1\pmod{4}, \\
\mathbb{Z}/4\mathbb{Z} & \text{if } p\equiv 3 \pmod{4}. \\
\end{cases}
\] 
\end{proposition}
\begin{proof}
By \cite[Corollary A.19]{DGNO1} $s\W_{pt}$ is a surjective image of
the restriction of homomorphism \eqref{map S} on the subgroup 
\[
\W_{pt}  = \bigoplus_{p \text{ is prime}}\, \W_{pt}(p) \subset \W. 
\]
The intersection of $\W_{pt}$ with $\mbox{Ker}(S)=\W_{Ising}$ is contained
in $\W_{pt}(2)\cong \mathbb{Z}/8\mathbb{Z} \oplus \mathbb{Z}/2\mathbb{Z}$ and is isomorphic to $\mathbb{Z}/8\mathbb{Z}$,
whence $s\W_{pt}(2) \cong  \mathbb{Z}/2\mathbb{Z}$ and $s\W_{pt}(p) \cong  \W_{pt}(p)$ for $p>2$.
\end{proof}

\begin{corollary}
\label{no odd in sW}
The group $s\W$ is $2$-primary, i.e., it has no non-trivial elements of odd order. 
\end{corollary}
\begin{proof}
Let $\C$ be a completely anisotropic slightly degenerate
braided fusion category. Then $\C \cong \C_{pt} \bts \C_{pt}'$
and the category $\C_{pt}'$ satisfies conditions of Corollary~\ref{slight property S}.
So if the order of $ [\C]$ in $s\W$ is finite then the order of 
$[\C_{pt}']$ in $s\W$ is at most $2$. 
We have
\[
[\C] = [\C_{pt}]  [\C_{pt}'].
\]
By Proposition~\ref{sWpt}  the order of $[\C_{pt}]$
in $s\W$ is even. Therefore, the order of $[\C]$ is even.
\end{proof}

Let $s\W_{2}$ (respectively, $s\W_{\infty}$) denote the subgroups of $s\W$
generated by the Witt classes of completely anisotropic $s$-simple   braided fusion categories of the the Witt order $2$
(respectively, of  infinite Witt order).

\begin{proposition}
\label{sW decomposition}
We have
\begin{equation}
\label{sW=sum}
s\W =  s\W_{pt} \bigoplus s\W_2 \bigoplus s\W_\infty.
\end{equation}
The subgroup $s\W_2$ is an elementary Abelian $2$-group and the subgroup
$s\W_\infty$ is a free Abelian group of countable rank.
\end{proposition}
\begin{proof}
That $s\W =  s\W_{pt} + s\W_2 + s\W_\infty$ is a consequence of Theorem~\ref{decomposition theorem}
and Corollary~\ref{slight property S}.
By Corollary~\ref{relations in sW} the only non-trivial relations in $s\W$ involving the classes of s-simple
categories are of the form $2[\C]=0$ for $[\C]\in s\W_2$. This proves that the sum is direct
and that $s\W_\infty$ is free. By \cite[Section 6.4]{DMNO} the Witt group $\W$ contains 
a free Abelian subgroup of countable rank. The homomorphism $S$  embeds this
subgroup into $s\W_\infty$, so the latter has a countable rank as well. 
\end{proof}

\begin{corollary}
\label{venets trudov}
The group $\W$ is $2$-primary.  The maximal finite order of an element
of $\W$ is $32$.
\end{corollary}
\begin{proof}
Let $\C$ be a
completely anisotropic non-degenerate braided fusion category.
If the order of  $[\C]$ in $\W$ is odd then $s\W$ contains a non-trivial
element of odd order, which is impossible by Corollary~\ref{no odd in sW}. Thus,
$\W$ is $2$-primary.

Suppose that the order of $[\C]$ in $\W$ is finite and $\geq 64$. 
We  have $\C \cong \mathcal{P} \bt \C_1$, where $\mathcal{P}$ is a pointed non-degenerate
braided fusion category and $\C_1$ is non-degenerate and such that
$(\C_1)_{pt}=\Vec$ or $(\C_1)_{pt}=\sVec$.  Then the order 
of $[\C_1]$ in $\W$ is $\geq 64$ and 
the order of $S([\C_1])=[\C_1 \bt \sVec]$ in $s\W$ is $\geq \frac{64}{16}=4$.  
By Theorem~\ref{sW decomposition} we have  $[\C_1]\in s\W_2$
so the order of $[\C_1]$ is at most $2$, a contradiction.

Note that $\W$ does contain an element of order $32$, namely
the Witt class of $\C(sl(2),\, 6)$ (see \cite[Section 6.4]{DMNO}).
\end{proof}

\begin{remark}
The subgroup $\sW_2$ is non-zero.
Indeed, it contains non-zero elements  $S([\C(so(2n+1),\,2n+1)]),\, n\geq 1$.  
From the existence of conformal embeddings  $so(m)_m \times so(m)_m \subset so(m^2)_1$
we see that  the Witt classes $[\C(so(2n+1),\,2n+1)]$ are square roots of  
the classes of Ising categories. Hence, they all have order $32$ in the Witt group $\W$. 
Therefore,  $S([\C(so(2n+1),\,2n+1)])\in \sW_2$ by Corollary~\ref{slight property S}.

We conjecture that the classes  $S([\C(so(2n+1),\,2n+1)]),\, n\geq 1,$ are pairwise
Witt non-equivalent in $s\W$,  so that  $\sW_2$  has an infinite order.
\end{remark}

\subsection{The subgroup of  $\W$ generated by classes $[\C(sl(2),\, k)],\, k\geq 1$}

We use the notations and observations made in \cite[Section 6.4]{DMNO}. 
For the basic facts about categories $\C(sl(2),\, k)$ see section \ref{Prelim-affine}.

It follows from the fusion rules of $\C(sl(2),\, k)$ that the object $[k]$ is invertible of order 2: 
$[k]\ot [k] = [0]$ and it generates $\C(sl(2),\, k)_{pt}$  (i.e., it is the only non-trivial invertible object). Its tensor product with 
simple objects has the form  $[k]\ot [i] = [k-i]$. 
The centralizer $\C(sl(2),\, k)_+$ of $[k]$ in $\C(sl(2),\, k)$ is the full subcategory with simple objects $[2s],\ s=0,...,2l+1$ of ``even spin" .
The braiding of $[k]$ with itself is well known: 
$$c_{[k],[k]} = \theta_{[k]} = e^{2\pi i\frac{k}{4}}\id_{[0]}.$$ 

Consider three different  cases of the level : $k=2l+1$, $k=4l$, and $k=4l+2$.

For $k=2l+1$ the pointed part $\C(sl(2),\, k)_{pt}$ is non-degenerate. By M\"uger's decomposition theorem \cite{Mu2}
$$\C(sl(2),\, k)\simeq \C(sl(2),\, k)_{pt}\boxtimes\C(sl(2),\, k)_{+}.$$
It can be seen by looking at the fusion rules that $\C(sl(2),\, k)_{+}$ is simple. It is also completely anisotropic. 
Note that $\C(sl(2),\, 1)_{pt}=\C(sl(2),\, 1)$ and its class has order 8 in the Witt group $\W$: 
\beq\lb{sl21}[\C(sl(2),\, 1)]^8=1.\eeq
Moreover we have the following relation between Witt classes $$[\C(sl(2),\, 2l+1)_{pt}] = [\C(sl(2),\, 1)]^{(-1)^l}.$$ 

For $k=4l$ the pointed part $\C(sl(2),\, k)_{pt}$ is Tannakian. Let $A = [0]\oplus [k]$ be the regular etale algebra. It is known that the category $\C(sl(2),\, 4l)^0_A$ of dyslectic $A$-modules is simple. It is also known that this category is completely anisotropic if $l\not=2,7$. For $l=2$ and $l=7$ one has (see \cite{DMNO})
\beq\lb{sl28}[\C(sl(2),\, 8)] = [\C(sl(2),\, 3)_+]^{-2} = [\C(sl(2),\, 3)]^{-2}[\C(sl(2),\, 1)]^2,
\eeq 
\beq\lb{sl228}[\C(sl(2),\, 28)] = [\C(sl(2),\, 3)_+] = [\C(sl(2),\, 3)][\C(sl(2),\, 1)]^{-1}.
\eeq
Note also that $\C(sl(2),\, 4l)^0_A$ is pointed only for $l=1$ and $\C(sl(2),\, 4)^0_A$ is of Witt order 4:
\beq\lb{sl24}[\C(sl(2),\, 4)]^4=1.\eeq

For $k=4l+2$ the pointed part $\C(sl(2),\, k)_{pt}$ is equivalent to $\sVec$.
The centraliser $\C(sl(2),\, 4l+2)_+$ is a slightly degenerate category. By inspecting fusion rules it becomes clear that $\C(sl(2),\, 4l+2)_+$ is s-simple. It is non-trivial (i.e., does not coincide with $\sVec$) for $l>0$. Triviality of $\C(sl(2),\, 2)_+$ corresponds to the fact that $\C(sl(2),\, 2)$ is an Ising category. In particular 
\beq\lb{sl26}[\C(sl(2),\, 2)]^{16}=1.\eeq
It is also completely anisotropic for all $l$ except $l=2$ and it was shown in \cite{DMNO} that 
\beq\lb{sl210}[\C(sl(2),\, 10)] = [\C(sl(2),\, 2)]^7.
\eeq
For $l\not=2$ we have two possibilities: either $\C(sl(2),\, 4l+2)_+$ is of infinite order in $\sW$ or $\C(sl(2),\, 4l+2)_+$ is braided equivalent to its reverse $\C(sl(2),\, k)^{rev}_+$. If $F:\C(sl(2),\, k)_+\to\C(sl(2),\, k)^{rev}_+$ is a braided equivalence, then $F([j])$ is either $[j]$ or $[k-j]$ (this follows from the condition $\FPdim(F([j])) = \FPdim([j])$). Together with the condition $\theta_{F([j])} = \theta^{-1}_{[j]}$ it gives $\theta_{[j]} = \pm\theta_{[j]}^{-1}$ for all $j=2s$. This is equivalent to saying that $\frac{s(s+1)}{2l}$ is an integer for any $s=0,...,2l+1$. Clearly this can happen only for $l=1$. Indeed, as was shown in \cite{DMNO}, one has the following relation between classes in $\W$:
\beq\lb{sl26+}[\C(sl(2),\, 6)]^2 = [\C(sl(2),\, 2)]^3,\eeq 
which implies that the class of $\C(sl(2),\, 6)_+$ in $\sW$ has order 2. For all other values of $l$ the class of the category $\C(sl(2),\, 4l+2)_+$ is of infinite order in $\sW$.

So far we have collected simple, completely anisotropic, non-degenerate categories
$$\C(sl(2),\, 2l+1)_+,\quad l\geq 1,$$
$$\C(sl(2),\, 4l)_A^0,\quad l\geq 3,\ l\not=7$$
and s-simple, completely anisotropic, slightly degenerate categories
$$\C(sl(2),\, 4l+2)_+,\quad l\geq 3.$$

By looking at the  Frobenius-Perron dimensions we can see that all categories in the third series are different.
It is easy to see by looking at the multiplicative central charges that the first two series have empty  intersection and that all categories in the series are different. Moreover their corresponding slightly degenerate categories of the form $\C\boxtimes\sVec$ are s-simple, completely anisotropic and all different to each other. Finally none of them can appear in the third series (simply because in contrast to the categories in the third series their symmetric centers split out). 
Thus they are  generators of the torsion free part of the subgroup of $\W$ generated by classes $[\C(sl(2),\, k)]$. 

Thus we have established the following.
\bth
All relations between the classes $[\C(sl(2),\, k)]$ in the Witt group $\W$ follow from the relations (\ref{sl21} -
\ref{sl26+}). 
\eth

\bpr
Let $I$ be a finite subset of  the set of odd positive integers.
The category $\boxtimes_{k\in I}\C(sl(2),\, k)_+$ is completely anisotropic.
\epr
\bpf
It is known that categories $\boxtimes_{k\in I}\C(sl(2),\, k)_+$ are completely anisotropic \cite{KiO}.
By Corollary~\ref{relations in sW} it suffices to show that 
\[
\C(sl(2),\, k)_+ \not\cong (\C(sl(2),\, k')_+)^\rev
\]
for any odd $k\neq k'$. This is clear from looking at the Frobenius-Perron 
dimensions of categories $\C(sl(2),\, k)_+$.
\epf


\end{section}

\bibliographystyle{ams-alpha}

\end{document}